\documentclass[reqno,a4paper,12pt]{article}
\usepackage{amsmath,amsthm,amssymb}
\usepackage[T1]{fontenc}
\usepackage{enumerate}
\usepackage{bbm}
\usepackage{bm}
\usepackage{mathtools}
\usepackage{mathrsfs}
\usepackage{nicefrac}
\usepackage{accents}
\usepackage[pagewise]{lineno}
\usepackage{paracol}
\usepackage[a4paper]{geometry}
\usepackage{tikz}
\tikzset{every picture/.style={line width=0.75pt}} %set default line width to 0.75pt  
\usepackage[sort&compress,numbers]{natbib}
\usepackage[hypertexnames=false, hidelinks, bookmarksnumbered=true]{hyperref}    
\usepackage{cleveref}  

\DeclareMathOperator{\lcm}{lcm}
\DeclareMathOperator{\supp}{supp}

\newtheorem{theorem}{Theorem}[section]
\newtheorem{lemma}[theorem]{Lemma}

\newtheorem{proposition}[theorem]{Proposition}
\newtheorem{Thx}{Theorem}

\newtheorem{Corx}[Thx]{Corollary}

\newtheorem{Propx}[Thx]{Proposition}

\theoremstyle{definition}
\newtheorem{definition}[theorem]{Definition}

\newtheorem{remark}[theorem]{Remark}
\newtheorem{example}[theorem]{Example}
\newtheorem{conjecture}{Conjecture}

\newcommand{\divides}{\mid}

\newcommand{\Per}{\operatorname{Per}\nolimits}

\newcommand{\cM}{{\cal M}}

\newcommand{\Z}{{\mathbb{Z}}}
\newcommand{\N}{{\mathbb{N}}}

\newcommand{\vep}{\varepsilon}

\newcommand{\sB}{{\mathscr{B}}}

\makeatletter
\def\moverlay{\mathpalette\mov@rlay}
\def\mov@rlay#1#2{\leavevmode\vtop{%
   \baselineskip\z@skip \lineskiplimit-\maxdimen
   \ialign{\hfil$\m@th#1##$\hfil\cr#2\crcr}}}
\newcommand{\charfusion}[3][\mathord]{
    #1{\ifx#1\mathop\vphantom{#2}\fi
        \mathpalette\mov@rlay{#2\cr#3}
      }
    \ifx#1\mathop\expandafter\displaylimits\fi}
\makeatother

\makeatletter
\newcommand{\subjclass}[2][2020]{%
  \let\@oldtitle\@title%
  \gdef\@title{\@oldtitle\footnotetext{#1 \emph{Mathematics subject classification.} #2}}%
}
\newcommand{\keywords}[1]{%
  \let\@@oldtitle\@title%
  \gdef\@title{\@@oldtitle\footnotetext{\emph{Key words and phrases.} #1.}}%
}
\makeatother

\providecommand{\noopsort}[1]{}

\makeatother
\begin{document}
\author{Aurelia Dymek \and Joanna Ku\l{}aga-Przymus \and Daniel Sell}
\title{Invariant measures for $\mathscr{B}$-free systems revisited}
\author{Aurelia Dymek \and Joanna Ku\l{}aga-Przymus \and Daniel Sell}
\title{Invariant measures for $\mathscr{B}$-free systems revisited}
\date{January 10, 2024}
\maketitle

\begin{abstract}
\noindent
For $ \mathscr{B} \subseteq \mathbb{N} $, the $ \mathscr{B} $-free subshift $ X_{\eta} $ is the orbit closure of the characteristic function of the set of $ \mathscr{B} $-free integers. We show that many results about invariant measures and entropy, previously only known for the hereditary closure of $ X_{\eta} $, have their analogues for $ X_{\eta} $ as well. In particular, we settle in the affirmative a conjecture of Keller about a description of such measures ([\textsc{Keller, G.} Generalized heredity in $\mathcal B$-free systems. \emph{Stoch. Dyn. 21}, 3 (2021), Paper No. 2140008]). A central assumption in our work is that $ \eta^{*} $ (the Toeplitz sequence that generates the unique minimal component of $ X_{\eta} $) is regular. From this we obtain natural periodic approximations that we frequently use in our proofs to bound the elements in $ X_{\eta} $ from above and below.
\end{abstract}

\tableofcontents
\section{Introduction}
\subsection{Background}
Given $\mathscr{B}\subseteq \N$, consider the corresponding set $\mathcal{M}_\mathscr{B}=\bigcup_{b\in\mathscr{B}}b\Z$ of the \emph{multiples} of $\mathscr{B}$ and its complement $\mathcal{F}_\mathscr{B}=\Z\setminus\mathcal{M}_\mathscr{B}$, i.e.\ the set of \emph{$\mathscr{B}$-free integers}. We study the dynamics of $\eta=\mathbf{1}_{\mathcal{F}_\mathscr{B}}\in \{0,1\}^\Z$, i.e.\ of the orbit closure $X_\eta$ of $\eta$ under the left shift $\sigma$. The motivation for such studies goes back to the 1930's, when sets of multiples were investigated from the \emph{number-theoretic perspective} by Besicovitch, Chowla, Erd\H{o}s and others (see~\cite{MR1414678} and the references therein). In 2010 Sarnak~\cite{sarnak-lectures} suggested to study the \emph{dynamics} of the square-free system, i.e.~$X_\eta$ corresponding to $\mathscr{B}$ being the set of squares of all primes. In this case $\eta|_\N$ is the square of the M\"obius function $\mu$, and the aim was to gain more knowledge about the M\"obius function itself. He formulated a certain ``program'' for $\mu^2$ and indicated how to prove the statements about $\mu^2$. Without going into details, there was a list of properties related both to measure-theoretic and topological dynamics of $X_{\mu^2}$. A natural question arose whether analogous results are true for other sets $\mathscr{B}$. The dynamics of $X_\eta$ was studied systematically for the first time in~\cite{MR3428961} in the Erd\H{o}s case, i.e.\ when $\mathscr{B}$ is infinite and pairwise coprime with $\sum_{b\in\mathscr{B}}\nicefrac{1}{b}<\infty$. In this case the properties of $X_\eta$ resemble the properties of $X_{\mu^2}$. In particular, $X_\eta$ is \emph{hereditary}, i.e.\ if $x\in X_\eta$ and $y\in\{0,1\}^\Z$ is such that $y\leq x$ coordinatewise then $y\in X_\eta$. In fact, we have 
\[
X_\eta=X_\mathscr{B}:=\{x\in\{0,1\}^\Z : |\supp x \bmod b| \leq b-1\text{ for any }b\in\mathscr{B}\}
\] 
($X_\mathscr{B}$ is called the $\mathscr{B}$-admissible subshift). When we relax the assumptions on $\mathscr{B}$, various things can happen to $X_\eta$, in particular it may no longer be hereditary. Thus, one often looks at its \emph{hereditary closure} $\widetilde{X}_\eta$, i.e.\ the smallest hereditary subshift containing~$X_\eta$. Such general $\mathscr{B}$-free systems were studied in~\cite{MR3803141}. We may have
\(
X_\eta\subsetneq \widetilde{X}_\eta\subsetneq X_\mathscr{B}
\)
(see~\cite{MR3803141} for various examples).

In this paper, we concentrate on invariant measures on $X_\eta$. Let us give now some more detailed background related to this. In the Erd\H{o}s case, $\eta$ turns out to be a generic point for the so-called \emph{Mirsky measure}~\cite{MR3428961} denoted by $\nu_\eta$:
\[
\frac{1}{L}\sum_{\ell\leq L}\delta_{\sigma^\ell \eta} \to \nu_\eta
\]
(in this case, the above formula can be treated as the definition of $\nu_\eta$). In other words, the frequencies of $0$-$1$ blocks on $\eta$ exist (in the square-free case they were first studied by Mirsky~\cite{MR21566,MR28334}, hence the name). In general, $\eta$ might not be a generic point. However, it is quasi-generic along any sequence $(\ell_i)$ realizing the lower density of $\mathcal{M}_\mathscr{B}$ (i.e.\ such that $\lim_{i\to\infty}\frac{1}{\ell_i}|\mathcal{M}_\mathscr{B}\cap[1,\ell_i]|=\liminf_{L\to\infty}\frac{1}{L}|\mathcal{M}_\mathscr{B}\cap[1,L]| =: \underline{d}(\mathcal{M}_\mathscr{B})$). This is a consequence of the deep number-theoretic result of Davenport and Erd\H{o}s~\cite{MR43835} that the logarithmic density of $\mathcal{M}_\mathscr{B}$, i.e.\ $\boldsymbol{\delta}(\mathcal{M}_\mathscr{B})=\lim_{L\to\infty}\frac{1}{\ln L}\sum_{\ell\in \mathcal{M}_\mathscr{B}\cap[1,L]}\frac{1}{\ell}$ always exists and we have 
\begin{equation}\label{DE}
\boldsymbol{\delta}(\mathcal{M}_\mathscr{B})=\underline{d}(\mathcal{M}_\mathscr{B})=\lim_{K\to \infty}d(\mathcal{M}_{\mathscr{B}_K}),\ \text{ where }\mathscr{B}_K=\{b\in\mathscr{B} : b\leq K\}
\end{equation}
($d(A)$ for $A\subseteq \Z$ stands for the natural density: $d(A)=\lim_{L\to\infty}\frac{1}{L}|A\cap [1,L]|$). Again, we call the resulting measure the Mirsky measure and denote it by $\nu_\eta$: 
\[
\lim_{i\to\infty}\frac{1}{\ell_i}\sum_{\ell\leq \ell_i}\delta_{\sigma^\ell \eta}= \nu_\eta,
\] 
see~\cite{MR3803141}. The following problems, already asked by Sarnak in the square-free case, arise:
\begin{itemize}
\item Give a description of the set $\mathcal{P}(X_\eta)$ of all invariant measures on $X_\eta$.
\item Compute the topological entropy $h(X_\eta)$ of $X_\eta$.
\item Determine, whether $X_\eta$ is \emph{intrinsically ergodic}, i.e.\ whether it has only one measure of maximal entropy.
\end{itemize}
The solution to the second problem and the positive answer to the third one in the square-free case were given by Peckner~\cite{Pe}. However, the proof used the properties of the squares of primes and it was not clear if it can be extended to a more general setting. It turned out to be true:
\begin{equation}\label{isinterg}
\text{for any }\mathscr{B}\subseteq\N,\ \widetilde{X}_\eta \text{ is intrinsically ergodic}.
\end{equation}
This was proved in~\cite{MR3356811} in the Erd\H{o}s case (where $X_\eta=\widetilde{X}_\eta$) and then, in~\cite{MR3803141}, for all sets $\mathscr{B}\subseteq\N$. Moreover, the topological entropy $h(\widetilde{X}_\eta)$ of $\widetilde{X}_\eta$ is equal to the upper density of $\mathcal{F}_\mathscr{B}$:
\begin{equation}\label{entro} 
h(\widetilde{X}_{\eta})=\overline{d}:=\overline{d}(\mathcal{F}_\mathscr{B})
\end{equation}
and 
\begin{equation}\label{miaramax} 
\text{the measure of maximal entropy on }\widetilde{X}_\eta\text{ is of the form }M_\ast(\nu_\eta\otimes B_{1/2,1/2}),
\end{equation}
where $B_{1/2,1/2}$ is the symmetric Bernoulli measure on $\{0,1\}^\Z$ and $M\colon (\{0,1\}^\Z)^2\to \{0,1\}^\Z$ stands for the coordinatewise multiplication (in each case, the proofs were given in the corresponding paper covering the intrinsic ergodicity in the same class). We also have 
\begin{equation}\label{kiedytriv} 
h(\widetilde{X}_\eta)=0 \iff \mathcal{P}(\widetilde{X}_\eta)=\{\delta_{\boldsymbol{0}}\}\iff \widetilde{X}_\eta \text{ is uniquely ergodic}
\end{equation} 
(this is true, in general, for hereditary subshifts, for a proof see~\cite{MR3007694}). 

As for the set of invariant measures, it was shown in~\cite{MR3356811} that in the Erd\H{o}s case 
\begin{equation}\label{opisek} 
\mathcal{P}(\widetilde{X}_\eta)=\{M_\ast(\nu_\eta\vee\kappa) : \kappa\in\mathcal{P}(\{0,1\}^\Z)\},
\end{equation}
where $\nu_\eta\vee\kappa$ stands for any joining of $\nu_\eta$ and $\kappa$, i.e.\ any probability measure $\rho$ on $(\{0,1\}^\Z)^2$ invariant under $\sigma^{\times 2}$ whose projection onto the first coordinate equals $\nu_\eta$ and the projection onto the second coordinate equals $\kappa$. Later, in~\cite{MR3803141}, this result was extended to any set $\mathscr{B}\subseteq\N$.

Recall that a central role in the theory of $\mathscr{B}$-free systems is played by the notion of tautness~\cite{MR1414678}:
\[
\mathscr{B}\subseteq \N \text{ is taut if for every }b\in\mathscr{B}\text{ we have }\boldsymbol{\delta}(\mathcal{M}_{\mathscr{B}\setminus \{b\}})<\boldsymbol{\delta}(\mathcal{M}_\mathscr{B}).
\]
It was shown in~\cite{MR3803141} (see Theorem C therein) that for any $\mathscr{B}\subseteq\mathbb{N}$, there exists a unique taut set $\mathscr{B}'\subseteq\N$ with $\mathcal{M}_\mathscr{B}\subseteq \mathcal{M}_{\mathscr{B}'}$ and $\nu_\eta=\nu_{\eta'}$ (for more details about $\mathscr{B}'$, see Section~\ref{seB}). In fact, we have
\begin{equation}\label{rownosc} 
\mathcal{P}(\widetilde{X}_{\eta})=\mathcal{P}(\widetilde{X}_{\eta'}).
\end{equation}
Moreover, in~\cite{MR3803141} (see Corollary 4.35 and Corollary 9.1 therein), the following combinatorial result on taut sets was proved. Fix $\mathscr{B}\subseteq\mathbb{N}$ and a taut set $\mathscr{C}\subseteq \N$. Let $\eta_\mathscr{C}:=\mathbf{1}_{\mathcal{F}_\mathscr{C}}$. Then the following are equivalent:
\begin{multline}\label{AA}
\text{for each $b\in\mathscr{B}$ there exists $c\in\mathscr{C}$ such that $c\divides b$} \iff \eta_\mathscr{C}\leq \eta \iff \widetilde{X}_{\eta_\mathscr{C}}\subseteq \widetilde{X}_\eta\\ \iff \eta_\mathscr{C}\in\widetilde{X}_\eta \iff \nu_{\eta_\mathscr{C}}\in\mathcal{P}(\widetilde{X}_\eta) \iff \mathcal{P}(\widetilde{X}_{\eta_\mathscr{C}})\subseteq \mathcal{P}(\widetilde{X}_\eta).
\end{multline}
In particular, an immediate consequence of this result is a list of conditions equivalent to $\mathscr{B}=\mathscr{C}$, whenever both $\mathscr{B}$ and $\mathscr{C}$ are taut, see Theorem~L in~\cite{MR3803141}.

Last but not least, let us mention some results related to the subset $\mathcal{P}^e(\widetilde{X}_\eta)$ of the ergodic measures on $\widetilde{X}_\eta$. It was shown in~\cite{MR3589826} that $\mathcal{P}(\widetilde{X}_\eta)$ is a Poulsen simplex (i.e.\ a non-trivial simplex with dense subset of ergodic measures with respect to the weak-star topology) whenever $h(\widetilde{X}_\eta)>0$. Recall that the density of ergodic measures implies the arcwise connectedness of the set of invariant measures~\cite{MR500918} (the latter property was proved to hold in a hereditary setting wider than just $\mathscr{B}$-free systems in~\cite{MR3803667}). Recently, a yet stronger result was obtained by Konieczny, Kupsa and Kwietniak~\cite{MR4544150}: namely,
\begin{equation}\label{BBB} 
\text{the subset }\mathcal{P}^e(\widetilde{X}_\eta) \text{ of ergodic measures on }\widetilde{X}_\eta \text{ is entropy-dense in }\mathcal{P}(\widetilde{X}_\eta),
\end{equation}
i.e.\ for any $\mu\in\mathcal{P}(\widetilde{X}_\eta)$, there exist $\mu_n\in\mathcal{P}^e(\widetilde{X}_\eta)$ such that $\mu_n\to \mu$ weakly and the measure-theoretic entropies $h(\widetilde{X}_\eta,\sigma,\mu_n)$ of $(\widetilde{X}_\eta,\sigma,\mu_n)$ converge to the measure-theoretic entropy $h(\widetilde{X}_\eta,\sigma,\mu)$ of $(\widetilde{X}_\eta,\sigma,\mu)$.

Clearly, if $X_\eta$ is hereditary, all of the above results apply to $X_\eta=\widetilde{X}_\eta$. We study analogous questions and prove the analogues of \eqref{isinterg} - \eqref{BBB} for $X_\eta$ in the non-hereditary case. For a summary of our results, see Section~\ref{semain}.

\subsection{Notation and main objects}

\subsubsection{Dynamics}
We say that $(X,T)$ is a \emph{topological dynamical system} if $T$ is a homeomorphism of a compact metric space $X$. We equip $X$ with the Borel sigma-algebra. The set of all probability Borel $T$-invariant measures will be denoted by $\mathcal{P}(X,T)$ (or just $\mathcal{P}(X)$ if $T$ is clear from the context). The subset of \emph{ergodic measures} will be denoted by $\mathcal{P}^e(X,T)$ or $\mathcal{P}^e(X)$. For each choice of $\mu\in\mathcal{P}(X)$, the triple $(X,T,\mu)$ is called a \emph{measure-theoretic dynamical system}. Given two measure-theoretic dynamical systems $(X,T,\mu)$ and $(Y,S,\nu)$, we say that $(Y,S,\nu)$ is a \emph{factor} of $(X,T,\mu)$ whenever there exists a measurable map $\pi\colon X\to Y$ (defined $\mu$-a.e.) such that $\pi_\ast(\mu)=\nu$ and $\pi\circ T=S\circ \pi$ $\mu$-a.e.

Both in the measure-theoretic and in the topological setting there is a notion of \emph{entropy} that describes the complexity of a given system. The measure-theoretic entropy of $(X,T,\mu)$ is denoted by $h(X,T,\mu)$. We skip its lengthy definition and refer the reader, e.g., to~\cite{MR2809170}.  We will mostly deal with $0$-$1$ \emph{subshifts}, i.e.\ $(X,\sigma)$, where $X\subseteq \{0,1\}^\Z$ is closed and invariant under the \emph{left shift} $\sigma\colon \{0,1\}^\Z\to \{0,1\}^\Z$. In this case the \emph{topological entropy}, denoted by $h(X)$, is easy to define: if $p_n(X)$ is the number of distinct blocks of length $n$ appearing on $X$, then $h(X)=\lim_{n\to\infty}\frac{1}{n}\log_2 p_n(X)$. If $X$ is the orbit closure of $x\in\{0,1\}^\Z$, we will write $p_n(x)$ instead of $p_n(X)$. There is the following \emph{variational principle} (valid in general, not only for subshifts): $h(X,T)=\sup_{\mu\in\mathcal{P}(X,T)}h(X,T,\mu)$. In case of subshifts there is always at least one measure $\mu$ realizing the supremum from the variational principle. If this measure is unique, we say that $X$ is \emph{intrinsically ergodic}.

Given a topological dynamical system $(X,T)$ and a point $x\in X$, we say that $x$ is a \emph{generic point} for $\mu\in\mathcal{P}(X,T)$ if $\frac{1}{L}\sum_{\ell\leq L}\delta_{T^\ell x}\to \mu$ weakly. We say that $x\in X$ is \emph{quasi-generic for $\mu$ along $(\ell_i)$} if $\frac{1}{\ell_i}\sum_{\ell\leq \ell_i}\delta_{T^\ell x}\to \mu$ weakly.

Given two measure-theoretic dynamical systems $(X_i,T_i,\mu_i)$, $i=1,2$, we say that $\rho\in \mathcal{P}(X_1\times X_2,T_1\times T_2)$ (with $X_1\times X_2$ equipped with the product sigma-algebra) is a \emph{joining} of $(X_1,T_1,\mu_1)$ and $(X_2,T_2,\mu_2)$, whenever $\mu_i=(\pi_i)_\ast(\rho)$ for $i=1,2$ ($\pi_i$ will always denote the projection onto the $i$'th coordinate, we will also use similar notation for projections onto more than one coordinate). We write then $\rho=\mu_1\vee\mu_2$ or $\rho\in J((X_1,T_1,\mu_1),(X_2,T_2,\mu_2))$. We always have $\mu_1\otimes \mu_2\in J((X_1,T_1,\mu_1),(X_2,T_2,\mu_2))$. In fact, if $(Y_i,S_i,\nu_i)$ is a factor of $(X_i,T_i,\mu_i)$ via a factor map $F_i$, $i=1,2$ and $\rho = \nu_1\vee\nu_2$, then there exists $\hat{\rho}\in J((X_1,T_1,\mu_1),(X_2,T_2,\mu_2))$ such that $(Y_1\times Y_2,S_1\times S_2,\rho)$ is a factor of $(X_1\times X_2,T_1\times T_2,\hat{\rho})$ via $F_1\times F_2$ (for example the so-called \emph{relatively independent extension} of $\rho$ has such a property). Last but not least, for $S\colon (X_1,T_1,\mu_1)\to(X_2,T_2,\mu_2)$, we will denote by $\triangle_S$ the graph joining of $(X_2,T_2,\mu_2)$ and $(X_1,T_1,\mu_1)$ given by $\triangle_S(A_2\times A_1)=\mu_1(S^{-1}A_2\cap A_1)$ for any measurable $A_1\subseteq X_1$, $A_2\subseteq X_2$. (Note that usually $\triangle_S$ stands for the joining of $T_1$ and $T_2$ where the coordinates are written in the opposite order.) For more information on joinings we refer the reader to~\cite{MR1958753}.
\subsubsection{Toeplitz systems}
A sequence $x\in\{0,1\}^\Z$ is called \emph{Toeplitz} if for each $i\in\Z$ there exists $s\in\N$ such that $x(i+sk)=x(i)$ for all $k\in\Z$. \emph{A Toeplitz subshift} is the orbit closure of a Toeplitz sequence under the left shift. Any Toeplitz subshift is \emph{minimal} \cite{MR756807} (the orbit of each point is dense). For each symbol $a\in\{0,1\}$ and any $s\in\N$, we set
\[
\Per(x,a,s):=\{i\in\Z \ :\ x(i+s k)=a\text{ for all } k \in\Z\}.
\] 
The \emph{$s$-periodic part} of $x$ is defined to be the set of positions 
\[
\Per(x,s):=\Per(x,0,s)\cup\Per(x,1,s).
\]
A Toeplitz sequence $x$ is called \emph{regular} if
\[
\lim_{r\to \infty}d\left(\bigcup_{s\leq r}\operatorname{Per}(x,s)\right)=1.
\]
(Notice that this is equivalent to the usual definition via the so-called period structure.) The remaining Toeplitz sequences are called \emph{irregular}. For any regular Toeplitz sequence, the corresponding Toeplitz subshift is uniquely ergodic, see Theorem~5 in~\cite{MR255766}. For more information on Toeplitz sequences, we refer the reader for example to the survey~\cite{MR2180227}.

\subsubsection{\texorpdfstring{$\mathscr{B}$}{B}-free systems}
\label{seB}
Since the notation differs a bit between various papers related to $\mathscr{B}$-free systems that are crucial for this work, we need to make certain adjustments.
\paragraph{Subshifts}
First, let us recall the main subshifts that are of our interest. Given $\mathscr{B}\subseteq \N$, we consider
\[
X_\eta=\overline{\{\sigma^k\eta : k\in\Z\}}\subseteq X_\mathscr{B}=\{x\in\{0,1\}^\Z : |\text{supp }x\bmod b|\leq b-1 \text{ for each }b\in\mathscr{B}\},
\]
where $\text{supp }x=\{n\in\Z : x( n ) =1\}$ stands for the support of $x$. They are called the \emph{$\mathscr{B}$-free subshift \( X_{\eta} \)} and the \emph{$\mathscr{B}$-admissible subshift \( X_{\sB} \)} corresponding to the set $\mathscr{B}$. Moreover, the so-called \emph{hereditary closure} $\widetilde{X}_\eta$ of $X_\eta$ is given by $\widetilde{X}_\eta=M(X_\eta\times\{0,1\}^\Z)$, where $M\colon (\{0,1\}^\Z)^2\to\{0,1\}^\Z$ stands for the coordinatewise multiplication of sequences (this is equivalent to defining $\widetilde{X}_\eta$ as the smallest hereditary subshift containing $X_\eta$). Since $X_\mathscr{B}$ is hereditary, we have
\[
X_\eta\subseteq \widetilde{X}_\eta\subseteq X_\mathscr{B}.
\]
%Moreover, $X_\mathscr{B}$ is hereditary. 
Usually we will assume that $\mathscr{B}$ is \emph{primitive}, i.e.\ for any $b,b'\in\mathscr{B}$, if $b\divides b'$ then $b=b'$. This assumption has no influence on the studied dynamics since $\mathcal{M}_\mathscr{B}=\mathcal{M}_{\mathscr{B}^{prim}}$, where by $\mathscr{B}^{prim}$ we will denote the maximal primitive subset of $\mathscr{B}$.

In fact, there are also some other interesting subshifts of $X_\mathscr{B}$, that we will discuss in a later paragraph. Let us now give an overview of the most important classes of sets $\mathscr{B}$ appearing in the literature. We say that $\mathscr{B}\subseteq \N$ is:
\begin{itemize}
\item \emph{Erd\H{o}s} if $\mathscr{B}$ is infinite, pairwise coprime and $\sum_{b\in\mathscr{B}}\nicefrac{1}{b}<\infty$,
\item \emph{Besicovitch} if $d(\mathcal{M}_\mathscr{B})$ exists,
\item \emph{taut} if for every $b\in\mathscr{B}$, we have $\boldsymbol{\delta}(\mathcal{M}_{\mathscr{B}\setminus\{b\}})<\boldsymbol{\delta}(\mathcal{M}_\mathscr{B})$,
\item \emph{Behrend} if $\boldsymbol{\delta}(\mathcal{M}_\mathscr{B})=1$.
\end{itemize}
Recall (see Theorem 3.7 in~\cite{MR3803141}) that any non-trivial Behrend set contains an infinite pairwise coprime subset. Moreover, $\mathscr{B}$ is taut if and only if $c\mathcal{A}\not\subseteq \mathscr{B}$ for any Behrend set $\mathcal{A}$ and any \( c \in \N \), see~\cite{MR1414678}.

Given $\mathscr{B}\subseteq \mathbb{N}$, we can now define:
\begin{itemize}
\item $\mathscr{B}':=(\mathscr{B}\cup C)^{prim}$, where 
\[
C=\{c\in\N : c\mathcal{A}\subseteq \mathscr{B}\text{ for some Behrend set }\mathcal{A}\}.
\] 
$\mathscr{B}'$ is called the \emph{tautification of $\mathscr{B}$}, and it is the unique taut set such that $\nu_\eta=\nu_{\eta'}$ (see~\cite{MR3803141} and \cite{DKKu} for more details about $\mathscr{B}'$).
\item $\mathscr{B}^*:=(\mathscr{B}\cup D)^{prim}$, where 
\[
D=\{d\in \N : d\mathcal{A}\subseteq\mathscr{B}\text{ for some infinite pairwise coprime set }\mathcal{A}\}.
\]
$\mathscr{B}^*$ corresponds to the unique minimal subshift $X_{\eta^*}$ of $X_\eta$ (see Corollary 5 in~\cite{MR4280951}). 
By Lemma 3 c) in~\cite{MR4280951}, $\mathscr{B}^*$ does not contain a scaled copy of an infinite pairwise coprime subset. Thus, $\mathscr{B}^*$ does not contain a scaled copy of a Behrend set and, hence, $\mathscr{B}^*$ is taut (for another proof see Lemma~3.7 in~\cite{MR3947636}). Moreover, $\eta^*$ is a Toeplitz sequence (see Theorem B in~\cite{MR3947636}) with a subsequence of $(\lcm(\mathscr{B}^*_K))_{K\geq 1}$ being its period structure, which in particular means that $\eta^*$ is a regular Toeplitz sequence if and only if
\begin{equation}\label{pozycje}
%\bigcup_{K\geq 1}\operatorname{Per}(\eta^*,\lcm(\mathscr{B}^*_K))=\Z
\lim_{ K \to \infty }{d}( \Z \setminus \operatorname{Per}( \eta^{*} ,\lcm( \mathscr{B}^*_{K} ) ) = 0 .
\end{equation}
(We won't need the notion of a period structure of a Toeplitz sequence, so let us skip it here and refer the reader to~\cite{MR2180227}).
 
\end{itemize}
We have
\begin{equation}\label{j3}
X_{\eta^*}\subseteq X_{\eta'}\subseteq X_\eta,
\end{equation}
see Remark 3.22 in \cite{DKKu} for the first inclusion, and (27) in \cite{MR4289651} for the second one. Note also that it was shown earlier that $X_{\eta^*}\subseteq X_\eta$, see Corollary 1.5 in \cite{MR4280951}. We have
\begin{equation}\label{zale}
(\mathscr{B}')^* = \mathscr{B}^*.
\end{equation}
Indeed, $X_{(\eta')^*}$ is the unique minimal subshift of $X_{\eta'}$, while $X_{\eta^*}$ is the unique minimal subshift of $X_\eta$. Hence, since $X_{\eta'}\subseteq X_\eta$, it follows that $X_{(\eta')^*}=X_{\eta^*}$. This is equivalent to~\eqref{zale} by  Theorem~L from~\cite{MR3803141}, cf.~\eqref{AA}.

\paragraph{Basic algebraic objects}
There are also certain important objects of algebraic nature related to $\mathscr{B}$:
\begin{itemize}
\item
the product group $G:=\prod_{b\in\mathscr{B}}\Z/b\Z$,
\item
the canonical embedding $\Delta\colon\Z \to G$ given by $\Delta(n)=(n,n,\dots)$,
\item
the subgroup $H:=\overline{\Delta(\Z)}$,
\item
the rotation $R=R_{\Delta(1)}\colon H\to H$ given by $R(h)=h+\Delta(1)$,
\item
the window associated to $\mathscr{B}$, given by \(
W:=\{h\in H : h_b\neq 0 \text{ for each }b\in\mathscr{B}\} \), and the closure of its interior, which we denote by \( \underline{W} := \overline{int(W)} \),
\item
the coding function $\varphi_A\colon H\to \{0,1\}^\Z$ for $A\subseteq H$, given by $\varphi_A(h)(n)=1\iff h+\Delta(n)\in A$; note that $\varphi_A\circ R=\sigma\circ \varphi_A$; in particular, we will use
\[
\varphi:=\varphi_W \text{ and }\underline{\varphi}:=\varphi_{\underline{W}};
\]
note that $\eta=\varphi(\Delta(0))$,
\item
the subset of admissible sequences with only one residue class mod each $b\in\mathscr{B}$ missing: 
\[
Y:=\{x\in\{0,1\}^Z : |\text{supp }x \bmod b|=b-1 \text{ for each }b\in\mathscr{B}\}\subseteq X_\mathscr{B},
\]
\item
the function $\theta\colon Y\to G$ ``reading'' the (unique) missing residue class mod each $b\in\mathscr{B}$, which is given by $\theta(y)=h\iff \text{supp }y\cap (b\Z-h_b)=\emptyset$ for $b\in\mathscr{B}$.
\end{itemize}

All these objects can be defined just as well for $\mathscr{B}'$ and $\mathscr{B}^*$. We will use the superscripts $'$ and ${}^*$ to indicate which of them we mean. For example we have
$H'=\overline{\Delta'(\Z)}$ where $\Delta'\colon \Z \to G'$ and similarly
\[
W':=\{h\in H' : h_b\neq 0\bmod b \text{ for each }b\in \mathscr{B}'\} \text{ and } \underline{W^{\prime}} = \overline{int(W^{\prime})}.
\]
Also, we will write $\varphi^{\prime}$ for $\varphi^{\prime}_{W'}$ and $\underline{\varphi^{\prime}}$ for $\varphi^{\prime}_{\underline{W'}}$. 
\begin{remark}
Notice that the meaning of $W'$ differs from the one used in~\cite{MR4280951}: Keller used $W'$ for $\overline{int(W)}$, which we denote as $\underline{W}$. 
\end{remark}

\paragraph{Group homomorphisms}
By Lemma~1.2 in~\cite{MR4280951}, there is a continuous surjective group homomorphism 
\[
\Gamma_{H,H^*}\colon H\to H^*
\] 
given by the unique continuous extension of the map \( \Delta(n) \mapsto \Delta^*(n) \) to \( H \). In fact, the following lemma provides a direct formula for $\Gamma_{H,H^*}$ (by the definition of $\mathscr{B}^*$, for each $b^*$ there exists $b\in\mathscr{B}$ such that $b^*\divides b$).
\begin{lemma}\label{W3}
Let $h\in H$. Then
$\Gamma_{H,H^*}(h)_{b^*}=h_b\bmod b^*$ for any $b\in\sB$ and any $b^*\in\sB^*$ such that $b^*\mid b$. In particular, $\Gamma_{H,H^*}(h)_b=h_b$ for any $b\in\sB\cap\sB^*$.
\end{lemma}
\begin{proof}
Let $(n_k)_{k\geq1}$ be such that $\lim_{k\to\infty}\Delta(n_k) = h$. Fix $ b \in \sB $. Then $ n_k\bmod b = h_b $ for sufficiently large $k\geq1$. Therefore, $n_k\bmod b^* = h_b\bmod b^* $ for any $b^*$ such that $b^*\mid b$. The assertion follows by the continuity of $\Gamma_{H,H^*}$.
\end{proof}

Moreover, it was shown in Lemma~1.4 in~\cite{MR4280951} that
\begin{equation}\label{mapsto}
\Gamma_{H, H^{*}}( \underline{W} )=W^{*} \text{ and } \Gamma_{H, H^{*}}(H\setminus \underline{W})=H^{*} \setminus W^{*}.
\end{equation}
It follows that
\begin{equation}\label{eqn:PhiGamma}
\underline{\varphi}(h)=\varphi^*(\Gamma_{H,H^*}(h)).
\end{equation}
Indeed, \( \varphi^*(\Gamma_{H,H^{*}}(h))( n ) = 1 \) if and only if \( \Gamma_{H,H^{*}}( h ) + \Delta^{*}(n) = \Gamma_{H,H^{*}}( h + \Delta(n) ) \in W^{*} \), which is equivalent to \( h + \Delta(n) \in \underline{W} \) due to \eqref{mapsto}.

\paragraph{More subshifts}
We will also need:
\[
X_\varphi:=\overline{\varphi(H)}
\]
and
\begin{equation}\label{m1}
[\underline{\varphi},\varphi]:=\{x\in\{0,1\}^\Z : \underline{\varphi}(h)\leq x\leq \varphi(h)\text{ for some }h\in H\}.
\end{equation}
The subshift $X_\varphi$ first appeared in~\cite{MR3784254} (under the name $\mathcal{M}_W^G$) and was later studied in~\cite{MR3947636}. The set $[\underline{\varphi},\varphi]$ that may not be a subshift (it is $\sigma$-invariant, but is not necessarily closed) was introduced in~\cite{MR4280951}. If $\mathscr{B}$ is primitive then $\varphi(H)\subseteq [\underline{\varphi},\varphi]\subseteq X_\varphi$ by Theorem~1.1 in~\cite{MR4280951}, so in particular
\begin{equation}\label{j1}
X_\varphi=\overline{[\underline{\varphi},\varphi]}.
\end{equation} 
Moreover, if $\mathscr{B}$ is taut then by Corollary 1.2 in~\cite{MR4280951} we have
\begin{equation}\label{j2}
X_\eta=X_\varphi=\overline{[\underline{\varphi},\varphi]}.
\end{equation}

Similar notation to the one in~\eqref{m1} will be used for sequences. Given $w,x\in \{0,1\}^\Z$, we set
\[
[w,x]:=\{\sigma^m y\in \{0,1\}^\Z :  w\leq y \leq  x,\ m\in\Z\}.
\]
Again, this may not be a subshift; one can consider its closure $\overline{[w,x]}$ if necessary.
\begin{remark}
\label{rem:CodomainY}
Let us comment here on the codomain of $\theta$. Since $\theta$ is defined on whole $Y$, in general we cannot say much more than that $\theta(y)\in G$. It was shown in Remark~2.45 in~\cite{MR3803141} that $\theta(Y\cap\widetilde{X}_\eta)\subseteq H$. However, this is not sufficient for us: we need to think of $\theta$ as of a function from (at least) $Y\cap [\underline{\varphi},\varphi]$ to $H$. We will show that
\[ \theta(Y\cap \widetilde{X_\varphi}) = \theta( Y \cap X_{\varphi} ) \subseteq H. \]
In the first equality, ``\( \supseteq \)'' follows from \( \widetilde{X}_\varphi \supseteq X_{\varphi} \). For the converse inclusion, consider \( y \in Y \cap \widetilde{X}_{\varphi}  \) and \( x \in X_{\varphi} = \overline{\varphi(H)} \) with \( y \leq x \). Notice that \( \overline{\varphi(H)} \subseteq X_{\sB} \) since \( \supp \varphi( h ) \) misses the residue class \(-h_{b}\) modulo \( b \) and \( X_{\sB} \) is closed. Thus, \( \supp x \) misses at least one residue class for each \( b \in \sB \). Due to \( y \in Y \) and \( y \leq x \), the support of \( x \) misses exactly one residue class for each \( b \), namely the same as \( \supp y \). This yields \( \theta(y) = \theta( x ) \in \theta( Y \cap X_{\varphi} ) \).

To see \( \theta( Y \cap X_{\varphi} ) \subseteq H \), we fix \( b \in \sB \) and \( x \in X_{ \varphi } = \overline{ \varphi( H ) } \). Then there exists a sequence \( ( \varphi(h_{n}) ) \) which converges to \( x \), and (by definition) we have \( \varphi(h_{n}) |_{ -(h_{n})_{b} + b \Z } = 0 \). Since \( H \) is compact, we can assume that \( (h_{n}) \) has a limit \( h \in H \). In particular, there exists \( n_{0} \in \N \) with \( (h_{n})_{b} = h_{b} \) for all \( n \geq n_{0} \). This yields \( \varphi(h_{n}) |_{ -h_{b} + b \Z } = 0  \) for all \( n \geq n_{0} \) and thus \( x |_{ -h_{b} + b \Z } = 0 \). For \( x \in Y \cap X_{ \varphi } \), it follows that \( -h_{b} \) is the unique residue class modulo \( b \) that \( \supp x \) misses. Since \( b \in \sB \) was arbitrary, we obtain \( \theta( x ) = h \in H \).
\end{remark}

\subsubsection{Dynamical diagrams}
The aim of this section is to introduce a certain language related to diagrams involving dynamical systems and factoring maps between them. It will allow us to summarize some of our results on diagrams, which, in turn, can help to understand the structure of some more complicated proofs since the diagrams are easier to ``glue together'' than the assertions written in the form of sentences. We will use the language of category theory. Namely, we consider the category where:
\begin{itemize}
\item
the \emph{objects} are triples of the form $(X,T,\mathcal{P}_X)$, where $(X,T)$ is a topological dynamical system and $\emptyset\neq\mathcal{P}_X\subseteq \mathcal{P}(X)$; if $\mathcal{P}_X=\mathcal{P}(X)$, we skip it and write $(X,T)$ instead of $(X,T,\mathcal{P}(X))$;
\item
a \emph{morphism} from $(X,T,\mathcal{P}_X)$ to $(Y,S,\mathcal{P}_Y)$ is a map $f\colon (X,T,\mathcal{P}_X)\to (Y,S,\mathcal{P}_Y)$ such that there exist $X_0\subseteq X$ with $\mu(X_0)=1$ for any $\mu\in\mathcal{P}_X$, $f\colon X_0\to Y$, $f_\ast(\mathcal{P}_X)\subseteq \mathcal{P}_Y$ and $S\circ f=f\circ T$ on $X_0$.
\end{itemize}
Any graph whose vertices are the above-defined objects and arrows denote morphisms is called a \emph{dynamical diagram}.

\begin{remark}\label{thesame}
We identify two morphisms $f,g\colon (X,T,\mathcal{P}_X)\to (Y,S,\mathcal{P}_Y)$, whenever $f$ and $g$ agree on a subset $X_0\subseteq X$ that is of full measure for every measure $\mu\in\mathcal{P}_X$.
\end{remark}
\begin{definition}
We define the \emph{composition} of morphisms $f\colon (X,T,\mathcal{P}_X)\to (Y,S,\mathcal{P}_Y)$ and $g\colon (Y,S,\mathcal{P}_Y)\to (Z,R,\mathcal{P}_Z)$ as the composition $g\circ f$. Notice that such a definition is correct in view of Remark~\ref{thesame}. Indeed, let $f\colon X_0\to Y$ and $g\colon Y_0\to Z$, where $\mu(X_0)=1$ for every $\mu\in\mathcal{P}_X$ and $\nu(Y_0)=1$ for every $\nu\in\mathcal{P}_Y$. Then the composition $g\circ f$ is defined on $X_0\cap f^{-1}(Y_0)$ and $\mu(X_0\cap f^{-1}(Y_0))=1$ for any $\mu\in\mathcal{P}_X$.
\end{definition}
%\begin{definition} A set of morphisms between $(X,T,\mathcal{P}_X)$ and $(Y,S,\mathcal{P}_Y)$ is called a \emph{dynamical diagram}.
%\end{definition}

\begin{definition}
We will say that a dynamical diagram \emph{commutes} if for any choice of $(X,T,\mathcal{P}_X)$ and $(Y,S,\mathcal{P}_Y)$ in this diagram the composition of morphisms along any path connecting $(X,T,\mathcal{P}_X)$ with $(Y,S,\mathcal{P}_Y)$ does not depend on the choice of the path, including the trivial (zero) path.
\end{definition}
\begin{remark}\label{identycznosc}
In the definition of commutativity, we implicitly assume that our diagram includes, for each vertex $(X,T,\mathcal{P}_X)$, the identity map $id\colon (X,T,\mathcal{P}_X)\to (X,T,\mathcal{P}_X)$. To increase the readability of the diagrams, we will skip the corresponding arrow in our figures. Notice that this means in particular that whenever a dynamical diagram of the form
\[
(X,T,\mathcal{P}_X)\mathrel{\mathop{\rightleftarrows}^{f}_{g}}(Y,S,\mathcal{P}_Y)
\] 
is commutative, then $g\circ f=id_X$ a.e.\ with respect to any $\mu\in\mathcal{P}_X$ and $f\circ g=id_Y$ a.e.\ with respect to any $\nu\in\mathcal{P}_Y$. Note that usually diagrams with loops do not appear in the context of commutative diagrams in category theory -- they will however appear in the present paper.
\end{remark}
\begin{remark}
\label{rem:linDiagComm}
In a commutative diagram for any pair of its vertices $(X,T,\mathcal{P}_X)$, $(Y,S,\mathcal{P}_Y)$ there is at most one morphism $f\colon (X,T,\mathcal{P}_X)\to (Y,S,\mathcal{P}_Y)$. Notice  also that any linear dynamical diagram is automatically commutative. (By a linear diagram we mean any diagram whose underlying undirected graph consinsts of vertices arranged in a line.) (The same applies to any dynamical diagram that is of the form of a directed tree -- a graph whose underlying undirected graph is a tree, i.e.\ a connected acyclic undirected graph.) 
\end{remark}

\begin{definition}
We will say that a morphism $f\colon (X,T,\mathcal{P}_X) \to (Y,S,\mathcal{P}_Y)$ is \emph{surjective} if $f_\ast(\mathcal{P}_X)=\mathcal{P}_Y$. We will say that a dynamical diagram is \emph{surjective} if every morphism in this diagram is surjective. If $(X,T,\mathcal{P}_X)\xrightarrow{f} (Y,S,\mathcal{P}_Y)$ is surjective, we will sometimes just say that (the morphism) $f$ is surjective. Notice that this notion is not the same as the surjectivity of the map $f\colon X\to Y$.
\end{definition}
\begin{remark}[cf.\ Remark~\ref{identycznosc}]\label{AAA}
Any commutative dynamical diagram that is a loop is automatically surjective. E.g.~if
\[
(X,T,\mathcal{P}_X)\mathrel{\mathop{\rightleftarrows}^{f}_{g}}(Y,S,\mathcal{P}_Y)
\] 
is a commutative dynamical diagram then it is surjective. Indeed, $\mathcal{P}_X=id_\ast(\mathcal{P}_X)=g_\ast(f_\ast(\mathcal{P}_X))\subseteq g_\ast(\mathcal{P}_Y)\subseteq \mathcal{P}_X$, so, in fact, $\mathcal{P}_X=g_\ast(\mathcal{P}_Y)$. By the same token, $\mathcal{P}_Y=f_\ast(\mathcal{P}_X)$. 
\end{remark}

\begin{example}\label{ega}
\item
\begin{enumerate}
\item Suppose that $\mathscr{B}\subseteq \mathbb{N}$ is taut. Then $\nu_\eta\in\mathcal{P}(X_\eta\cap Y)$ by Theorem H in~\cite{MR3803141}, so $\mathcal{P}(X_\eta\cap Y)\neq\emptyset$. Thus
\[
(X_\eta\cap Y,\sigma)\xrightarrow{\theta}(H,R)\xrightarrow{\varphi}(X_\eta,\sigma)
\]
is a dynamical diagram. Its subdiagram $(X_\eta\cap Y,\sigma)\xrightarrow{\theta}(H,R)$ is surjective (by the unique ergodicity of $(H,R)$, we have $\theta_\ast(\nu)=m_H\in\mathcal{P}(H)$ for any $\nu\in\mathcal{P}(X_\eta\cap Y)$), while $(H,R)\xrightarrow{\varphi}(X_\eta,\sigma)$ is not surjective unless $X_\eta$ is uniquely ergodic (cf.\ Corollary~\ref{zerowanie} in Section~\ref{mainresults}).
\item 
The dynamical diagram
\[
(\{0,1\}^\Z,\sigma)\mathrel{\mathop{\rightleftarrows}^{\sigma}_{\sigma}} (\{0,1\}^\Z,\sigma)
\]
does not commute: indeed, $\sigma\circ \sigma\neq id$ (except at the four fixed points of $\sigma^2$).
%Diagram \[(X_\eta\cap Y,\sigma)\mathrel{\mathop{\rightleftarrows}^{\theta}_{\varphi}}(H,R)\] does not commute a.e. First of all, we have $\varphi_\ast(m_H)=\nu_\eta$ and this measure is not always supported on $X_\eta\cap Y$. A way around this would be to assume that $\mathscr{B}$ is taut. But even in that case it is not necessarily true that $\varphi\circ \theta(x)=x$ a.e. (we only know that $\varphi(\theta(x))\geq x$ and it is easy to find examples without equality). 
Notice however, that if we equip each vertex with
$\emptyset\neq\mathcal{P}\subseteq\{\delta_{\mathbf{0}},\frac{1}{2}(\delta_{\ldots 1 0 1 0 1\ldots}+\delta_{\ldots 0 1 0 1 0\ldots}),\delta_{\mathbf{1}}\})$ then
\[
(\{0,1\}^\Z,\sigma,\mathcal{P})\mathrel{\mathop{\rightleftarrows}^{\sigma}_{\sigma}} (\{0,1\}^\Z,\sigma,\mathcal{P})
\]
becomes a commutative dynamical diagram (and thus it is surjective by Remark~\ref{AAA}).
\item If $\mathscr{B}\subseteq \N$ is taut then  
\[
(X_\eta,\sigma,\{\nu_\eta\})\mathrel{\mathop{\rightleftarrows}^{\varphi^{-1}}_{\varphi}}(H,R)
\] is a commutative dynamical diagram (and thus it is surjective by Remark~\ref{AAA}). Indeed, $\varphi\colon (H,R,m_H)\to (X_\eta,\sigma,\nu_\eta)$ is a measure-theoretic isomorphism, see~\cite{MR3428961} for the Erd\H{o}s case and Theorem F in~\cite{MR3803141} for the taut case. The map $\varphi^{-1}$ can be replaced with $\theta$ (recall that for taut $\mathscr{B}$, we have $\nu_\eta(X_\eta\cap Y)=1$, so $\theta$ is well-defined $\nu_\eta$-a.e.\ on $X_\eta$).
\item
The following diagram 
\[
((\{0,1\}^\Z)^2,\sigma^{\times 2},\{\nu_\eta\vee\kappa : \kappa\in \mathcal{P}(\{0,1\}^\Z)\})\xrightarrow{M} (\widetilde{X}_\eta,\sigma)
\]
is clearly a dynamical diagram (as $M\circ (\sigma\times \sigma)=\sigma\circ M$ everywhere). It is linear, hence commutative. Moreover, it is surjective by~\eqref{opisek}.
\end{enumerate}
\end{example}
%\begin{remark}
%Suppose that we have two dynamical diagrams, e.g., 
%\[
%\mathbf{D_1}:=[(X,T,\mathcal{P}_X)\xrightarrow{f}(Y,S,\mathcal{P}_Y)] \text{ and }\mathbf{D_2}:=[(Y,S,\mathcal{P}_Y)\xrightarrow{g} (Z,R,\mathcal{P}_Z)].
%\]
%Then $\mathbf{D_1},\mathbf{D_2}$ commute then the joint diagram
%\[
%\mathbf{D}=\mathbf{D_1}\vee\mathbf{D_2}:=[(X,T,\mathcal{P}_X)\xrightarrow{f}(Y,S,\mathcal{P}_Y)\xrightarrow{g} (Z,R,\mathcal{P}_Z)]
%\]
%also commutes. Moreover, if both $\mathbf{D_1}$ and $\mathbf{D_2}$ are surjective the also the joint diagram $\mathbf{D}$ is surjective. This remarks extends to more complex diagrams, provided that we glue together only one pair of vertices from each of them. This can be extended further to even more complicated situations where more edges are glued together, but then one needs to check the new arising commutativity relations, c.f.\ Example~\ref{ega} (part 2) above.
%\end{remark}

\subsection{Summary}\label{semain}
In this section we present our main results. They are divided into three groups:
\begin{itemize}
\item results about invariant measures, 
\item combinatorial results related to the notion of tautness,
\item entropy results. 
\end{itemize}
We also discuss how to interpret some of them in terms of dynamical diagrams and indicate the main steps in their proofs.

\subsubsection{Main results: invariant measures}\label{mainresults}
In~\cite{MR4280951}, Keller formulated a conjecture on the form of $\mathcal{P}(X_\eta)$. Let us restate it using our notation.
\begin{conjecture}[{Conjecture~1 in~\cite{MR4280951}}]
\label{hipoteza}
Let $\mathscr{B}\subseteq\N$ be such that $\eta^*$ is a regular Toeplitz sequence. Then for any $\nu\in \mathcal{P}(X_\varphi)$, there exists $\rho\in \mathcal{P}(H\times \{0,1\}^\Z,R\times \sigma)$ such that for any measurable $A\subseteq X_\varphi$
\[
\nu(A)=\int_{H\times \{0,1\}^\Z}\mathbf{1}_A(\underline{\varphi}(h)+x\cdot(\varphi(h)-\underline{\varphi}(h)))\, d\rho(h,x).
\]
In other words, for each $\nu\in\mathcal{P}(X_\varphi)$, we have $\nu=(M_H)_\ast(\rho)$ for some $\rho\in\mathcal{P}(H\times \{0,1\}^\Z, R\times \sigma)$, where $M_H\colon H\times \{0,1\}^\Z\to [\underline{\varphi},\varphi]$ is given by
\[ M_H(h,x)=\underline{\varphi}(h)+x\cdot(\varphi(h)-\underline{\varphi}(h)). \]
\end{conjecture}
Notice that each $\rho\in\mathcal{P}(H\times \{0,1\}^\Z,R\times \sigma)$ is a joining of $m_H$ with some measure $\kappa\in\mathcal{P}(\{0,1\}^\Z)$, i.e.\ $\rho=m_H\vee\kappa$. Our motivation for writing this paper was to prove the above conjecture. In fact, we will prove  not only that all $\sigma$-invariant measures on $X_\varphi$ are of the form $(M_H)_\ast(m_H\vee\kappa)$, but also that the opposite inclusion holds and that $\mathcal{P}(X_\eta)=\mathcal{P}(X_\varphi)$. Thus, we not only settle Keller's conjecture, but also answer his question from~\cite{MR4280951} about the existence of invariant measures supported on $X_{\varphi}\setminus X_\eta$ (there are no such measures). The following theorem that captures all of this is our main result.
\begin{Thx}\label{glowne}
For any $\mathscr{B}\subseteq \N$ such that $\eta^*$ is a regular Toeplitz sequence, we have
\[
\mathcal{P}(X_\eta)=\mathcal{P}(X_\varphi)=\{(M_H)_\ast(m_H\vee\kappa):\kappa\in \mathcal{P}(\{0,1\}^\Z)\}.
\]
\end{Thx}

An auxiliary result, used to prove Theorem~\ref{glowne}, but also interesting on its own, is another description of the set $\mathcal{P}(X_\eta)=\mathcal{P}(X_\varphi)$.
\begin{Thx}[cf.~\eqref{opisek}]\label{pomocnicze}
For any $\mathscr{B}\subseteq \N$ such that $\eta^*$ is a regular Toeplitz sequence, we have
\[
\mathcal{P}(X_\eta)=\mathcal{P}(X_\varphi)=\{N_\ast((\nu_{\eta*}\triangle\nu_\eta)\vee\kappa) : \kappa\in\mathcal{P}(\{0,1\}^{\Z})\},
\]
where $N\colon (\{0,1\}^{\Z})^3\to \{0,1\}^{\Z}$ is the map given by $N(w,x,y)=w+y\cdot(x-w)$ and $\nu_{\eta^*}\triangle\nu_\eta$ is the joining of $\nu_{\eta^*}$ with $\nu_\eta$ for which the pair $(\eta^*,\eta)$ is quasi-generic along any sequence~$(\ell_i)$ realizing the lower density of $\mathcal{M}_\mathscr{B}$.
\end{Thx}
\begin{remark}
Note that it is non-trivial that $(\eta^*,\eta)$ is quasi-generic under $\sigma\times\sigma$ along $(\ell_i)$ realizing the  lower density of $\mathcal{M}_\mathscr{B}$ -- this will be shown in course of the proof of Theorem~\ref{pomocnicze}. In fact, we will describe the limit measure, see Lemma~\ref{prop:PhiMHGeneric}. Notice also that once a pair $(x,y)\in\{0,1\}^\Z$ is quasi-generic under $\sigma\times\sigma$ for some measure $\rho$ then $\rho$ is $(\sigma\times\sigma)$-invariant. Moreover, $x$ and $y$  are quasi-generic (along the same subsequence) for the marginals of $\rho$ and thus, $\rho$ is joining of its marginals.
\end{remark}

\begin{Thx}[cf.\ \eqref{rownosc}, recall also~\eqref{j3}]\label{wnioC}
For any $\mathscr{B}\subseteq\N$ such that $\eta^*$ is a regular Toeplitz sequence, we have $\mathcal{P}(X_\eta)=\mathcal{P}(X_{\eta'})$.
\end{Thx}

\subsubsection{Main results: tautness and combinatorics}
\begin{Propx}[cf.~\eqref{AA}]\label{NNN}
Let $\mathscr{B}\subseteq\mathbb{N}$. Suppose that $\mathscr{C}\subseteq \N$ is taut. Then the following are equivalent:
\begin{paracol}{2}
\begin{enumerate}[(a)]
\item $(\forall_{b\in\mathscr{B}} \exists_{c\in\mathscr{C}}$\ $c\divides b)$ and $(\forall_{c\in\mathscr{C}}\exists_{b^*\in\mathscr{B}^*}\ b^*\divides c)$,
\item $\eta^*\leq \eta_\mathscr{C}\leq \eta$,
\item $X_{\eta_\mathscr{C}}\subseteq X_\eta$,
\item $\eta_\mathscr{C}\in X_\eta$,
\item $\nu_{\eta_\mathscr{C}}\in\mathcal{P}(X_\eta)$,
\item $\mathcal{P}(X_{\eta_\mathscr{C}})\subseteq \mathcal{P}(X_\eta)$,
\end{enumerate}
\switchcolumn
\begin{enumerate}[(a')]
\item $(\forall_{b'\in\mathscr{B}'} \exists_{c\in\mathscr{C}}$\ $c\divides b')$ and $(\forall_{c\in\mathscr{C}}\exists_{b^*\in\mathscr{B}^*}\ b^*\divides c)$,
\item $\eta^*\leq \eta_\mathscr{C}\leq \eta'$,
\item $X_{\eta_\mathscr{C}}\subseteq X_{\eta'}$,
\item $\eta_\mathscr{C}\in X_{\eta'}$,
\item $\nu_{\eta_\mathscr{C}}\in\mathcal{P}(X_{\eta'})$,
\item $\mathcal{P}(X_{\eta_\mathscr{C}})\subseteq \mathcal{P}(X_{\eta'})$. 
\end{enumerate}
\end{paracol}
\end{Propx}
Given $\mathscr{B}\subseteq \mathbb{N}$, let 
\[
Taut(\mathscr{B}):=\{\mathscr{C}\subseteq \mathbb{N} : \mathscr{C}\text{ is taut and }X_{\eta_\mathscr{C}}\subseteq X_\eta\}.
\]
Consider the following partial order $\prec$ on $Taut(\mathscr{B})$:
\[
\mathscr{C}_1 \prec \mathscr{C}_2 \iff X_{\eta_{\mathscr{C}_1}}\subseteq X_{\eta_{\mathscr{C}_2}}.
\]
Clearly, $\mathscr{B}',\mathscr{B}^*\in Taut(\mathscr{B})$. Moreover, $\mathscr{B^*}$ is the smallest element of $Taut(\mathscr{B})$. Indeed, if $\mathscr{C}\in Taut(\mathscr{B})$ then $X_{\eta^*}\subseteq X_{\eta_\mathscr{C}}$ since  $X_{\eta^*}$ is the unique minimal subset of $X_\eta$. As an immediate consequence of Proposition~\ref{NNN} (more precisely, by $(c)\iff (c')$), we obtain the following.
\begin{Corx}
For any $\mathscr{B}\subseteq \mathbb{N}$, $\mathscr{B}'$ is the largest element of $Taut(\mathscr{B})$ with respect to $\prec$.
\end{Corx}
\subsubsection{Main results: entropy}
Last, but not least, we prove some results on the entropy of $\mathscr{B}$-free systems.

\begin{Thx}[cf.~\eqref{entro}]\label{entropia}
For any $\mathscr{B}\subseteq \N$, we have $h(X_\eta)\geq \overline{d}-\overline{d}^*$, where $\overline{d}=\overline{d}(\mathcal{F}_\mathscr{B})$ and $\overline{d}^*=\overline{d}(\mathcal{F}_{\mathscr{B}^*})$. If additionally $X_{\eta^*}$ is uniquely ergodic  (in particular if $\eta^*$ is a regular Toeplitz sequence) then $h(X_\eta)=\overline{d}-d^*$, where $d^*=d(\mathcal{F}_{\mathscr{B}^*})$.
\end{Thx}
\begin{Corx}[cf.~\eqref{kiedytriv}]\label{zerowanie}
For any $\mathscr{B}\subseteq \N$ such that $\eta^*$ is a regular Toeplitz sequence, we have 
\[
h(X_\eta)=0 \iff \mathcal{P}(X_\eta)=\{\nu_{\eta}\}\iff X_\eta \text{ is uniquely ergodic}
\]
(note that if the above holds then $\nu_{\eta}=\nu_{\eta^*}$).
\end{Corx}
\begin{remark}
In Corollary~\ref{zerowanie} the second equivalence is true without any assumption on $\eta^*$. It seems open whether there exists $\mathscr{B}$ such that $\eta^*$ is an irregular Toeplitz sequence with $h(X_{\eta^*})=0$ and $\mathcal{P}(X_\eta)$ being not a singleton, cf.~Remark~\ref{piwnica} below.
\end{remark}
\begin{Thx}[cf.~\eqref{isinterg} and~\eqref{miaramax}]\label{thme}
For any $\mathscr{B}\subseteq\N$ such that $\eta^*$ is a regular Toeplitz sequence, the subshift $X_\eta$ is intrinsically ergodic. The measure of maximal entropy equals $N_\ast((\nu_{\eta^*}\triangle\nu_{\eta})\otimes B_{1/2,1/2})$.
\end{Thx}
\begin{Thx}[cf.~\eqref{BBB}]\label{dominik}
For any $\mathscr{B}\subseteq\N$ such that $\eta^*$ is a regular Toeplitz sequence, the ergodic measures are entropy-dense in $\mathcal{P}(X_\eta)$.
\end{Thx}

\subsubsection{Dynamical diagrams viewpoint}\label{asercje}
In this section we present a dynamical diagrams viewpoint on Theorem~\ref{pomocnicze}, Theorem~\ref{wnioC}, Theorem~\ref{glowne} and Theorem~\ref{thme}. The first three of these results can be formulated in terms of dynamical diagrams and the structure of the proofs also relies on this notion. As for Theorem~\ref{thme}, the dynamical diagrams serve as a tool in the proof.

\paragraph{On Theorem~\ref{pomocnicze} and Theorem~\ref{wnioC}}
These two results can be proved separately (Theorem~\ref{wnioC} is then a consequence of Theorem~\ref{pomocnicze}), however, there is a nice way to treat them together, which has the additional advantage of slightly shortening the proofs. Recall that by~\eqref{j3},~\eqref{j1} and~\eqref{j2}, we have
\[
\overline{[\underline{\varphi^{\prime}},\varphi^{\prime}]}=X_{\varphi^{\prime}}=X_{\eta'}\subseteq X_\eta \subseteq X_\varphi = \overline{[\underline{\varphi},\varphi]}.
\]
Moreover, under the extra assumption that $\eta^*$ is a regular Toeplitz sequence, by Remark~1.4 in \cite{MR4280951} and observing \eqref{zale} we have
\begin{equation}\label{etaregular}
\mathcal{P}(X_\varphi)=\mathcal{P}([\underline{\varphi},\varphi]) \text{ and }\mathcal{P}(X_{\varphi'})=\mathcal{P}([\underline{\varphi}',\varphi']).
\end{equation}

\begin{remark}\label{piwnica}
Equality~\eqref{etaregular} is actually the main reason for the extra assumption on $\eta^*$ in Keller's conjecture from~\cite{MR4280951} (see Remark~1.4 therein). In fact, this goes deeper. If $\eta^*$ is a regular Toeplitz sequence then $\mathcal{P}(X_{\eta^*})=\{\nu_{\eta^*}\}$, while when we drop the assumption on $\eta^*$, various things can happen to $\mathcal{P}(X_{\eta^*})$: it can be a singleton consisting only of $\nu_{\eta^*}$, see~Theorem~2 in~\cite{IrK} (even if $\eta^*$ is an irregular Toeplitz sequence!), but it can also contain some positive entropy measure, see~Theorem~1 in~\cite{IrK}. Thus, since $X_{\eta^*}\subseteq X_\eta$, we cannot expect to obtain a consistent description of $\mathcal{P}(X_\eta)$ without imposing any restrictions on $\eta^*$. We will use the fact that the Toeplitz sequence $\eta^*$ is regular very frequently in our proofs.
\end{remark}

Continuing our argument from \eqref{etaregular}, we obtain for a regular Toeplitz sequences \( \eta^{*} \) that
\begin{multline}\label{a}
\mathcal{P}([\underline{\varphi^{\prime}},\varphi^{\prime}])=\mathcal{P}(\overline{[\underline{\varphi^{\prime}},\varphi^{\prime}]})=\mathcal{P}(X_{\varphi^{\prime}})=\mathcal{P}(X_{\eta'})\\
\subseteq \mathcal{P}(X_\eta) \subseteq \mathcal{P}(X_\varphi) =\mathcal{P}(\overline{[\underline{\varphi},\varphi]})=\mathcal{P}([\underline{\varphi},\varphi]).
\end{multline}
Therefore, the assertions of Theorem~\ref{pomocnicze} and Theorem~\ref{wnioC} are equivalent to the following two inclusions:
\begin{equation}\label{zato1}
\mathcal{P}([\underline{\varphi},\varphi])\subseteq \{N_\ast((\nu_{\eta^*}\triangle\nu_\eta)
\vee\kappa) : \kappa\in\mathcal{P}(\{0,1\}^\Z)\} \subseteq \mathcal{P}(X_{\varphi'}).
\end{equation}
Consider the following diagram:
\begin{gather}\tag{$\mathbf{D}_{B,C}$}\label{diagBC}
\begin{tikzpicture}[baseline=(current  bounding  box.center)]
\node (T) at (0,1.5) {$(  (\{0,1\}^{\Z})^3,\sigma^{\times 3},\{(\nu_{\eta^*}\triangle\nu_\eta)\vee\kappa: \kappa\in\mathcal{P}(\{0,1\}^{\Z})\} )
$};
\node (M) at (0,0) {$([\underline{\varphi},\varphi],\sigma)$};
\node (B) at (0,-1.5) {$(X_{\varphi'},\sigma)$};
\draw[->]  (T) edge node[auto] {$N$} (M)
	       (M) edge node[auto] {$id$} (B);
\end{tikzpicture}
\end{gather}
and notice that the assertions of Theorem~\ref{pomocnicze} and Theorem~\ref{wnioC} are equivalent to~\eqref{diagBC} being a surjective commutative dynamical diagram. Indeed,
\begin{itemize}
\item \eqref{diagBC} is a dynamical diagram $\iff$ the maps $N$ and $id$ are morphisms $\implies$ the second inclusion in~\eqref{zato1} holds.
\end{itemize} 
Notice that by~\eqref{a}, the map $id$ is a morphism if and only if $\mathcal{P}([\underline{\varphi},\varphi])=\mathcal{P}(X_{\varphi'})$. Therefore, $id$ is then automatically surjective.
\begin{itemize}
\item \eqref{diagBC} is in addition surjective $\iff$ the morphism $N$ is surjective $\implies$ the first inclusion in~\eqref{zato1} holds.
\end{itemize}
Moreover, if both inclusions in~\eqref{zato1} hold, then by~\eqref{a} we obtain the equality $ \{N_\ast((\nu_{\eta^*}\triangle\nu_\eta) \vee\kappa) : \kappa\in\mathcal{P}(\{0,1\}^\Z)\} = \mathcal{P}([\underline{\varphi},\varphi]) = \mathcal{P}(X_{\varphi'}) $, which implies that \( N \) and \( id \) are surjective morphisms.

\paragraph{On Theorem~\ref{glowne}}
Having proved Theorem~\ref{pomocnicze} first, in order to prove Theorem~\ref{glowne} we will only need to show that
\[
\{N_\ast((\nu_{\eta*}\triangle\nu_\eta)\vee\kappa) : \kappa\in\mathcal{P}(\{0,1\}^{\Z})\}=\{(M_H)_\ast(m_H\vee\kappa):\kappa\in \mathcal{P}(\{0,1\}^\Z)\}.
\]
Let $\underline{\varphi}\otimes \varphi\colon H\to (\{0,1\}^\Z)^2$ and $(\underline{\varphi}\otimes \varphi)(h)=(\underline{\varphi}(h),\varphi(h))$. Consider the following diagram:
\begin{gather}\tag{$\mathbf{D}_A$}\label{zato2}
\begin{tikzpicture}[baseline=(current  bounding  box.center)]
\node (T) at (0,0) {$(H\times \{0,1\}^\Z,R\times \sigma)$};
\node (M) at (0,-1.5) {$ ((\{0,1\}^{\Z})^3,\sigma^{\times 3},\{(\nu_{\eta^*}\triangle\nu_\eta)\vee\kappa:\kappa\in\mathcal{P}(\{0,1\}^\Z)\})$};
\node (B) at (0,-3) {$(X_\eta,\sigma)$};
\node (TR) at (5,0) {};
\node (BR) at (5,-3) {};
\draw (T) -- (5,0) -- node[auto] {$M_H$} (5,-3);
\draw [->] (5,-3) -- (B);
\draw[->]  (T) edge node[auto] {$(\underline{\varphi}\otimes \varphi)\times id$} (M)
		   (M) edge node[auto] {$N$} (B);
\end{tikzpicture}
\end{gather}
Then:
\begin{itemize}
\item \eqref{zato2} is a commutative dynamical diagram $\implies$ 
\[
\{(M_H)_\ast(m_H\vee\kappa):\kappa\in \mathcal{P}(\{0,1\}^\Z)\}\subseteq \{N_\ast((\nu_{\eta*}\triangle\nu_\eta)\vee\kappa) : \kappa\in\mathcal{P}(\{0,1\}^{\Z})\}
\]
(indeed, by the commutativity, ``travelling'' via $M_H$ is the same as ``travelling'' first via $(\underline{\varphi}\otimes \varphi)\times id$ and then via $N$),
\item \eqref{zato2} is surjective $\implies$ 
\[
\{ N_\ast((\nu_{\eta*}\triangle\nu_\eta)\vee\kappa) : \kappa\in\mathcal{P}(\{0,1\}^{\Z}) \}\subseteq
 \{ (M_H)_\ast(m_H\vee\kappa):\kappa\in \mathcal{P}(\{0,1\}^\Z)\}
\]
(indeed, we can travel up from $N_\ast((\nu_{\eta*}\triangle\nu_\eta)\vee\kappa)$ to $(\nu_{\eta*}\triangle\nu_\eta)\vee\kappa$ via $N$, then again up by $(\underline{\varphi}\otimes \varphi)\times id$, i.e.\ use the surjectivity of $(\underline{\varphi}\otimes \varphi)\times id$ and finally use that $M_H=N\circ((\underline{\varphi}\otimes \varphi)\times id)$ as \eqref{zato2} commutes).
\end{itemize}
In other words, the assertion of Theorem~\ref{glowne} follows from Theorem~\ref{pomocnicze} and the commutativity and the surjectivity of \eqref{zato2}.

In fact, Theorem~\ref{glowne}, Theorem~\ref{pomocnicze} and Theorem~\ref{wnioC} can be summarized using a single diagram, namely
\[
\begin{tikzpicture}[baseline=(current  bounding  box.center)]
\node (2) at (0,-1.5) {$(H\times \{0,1\}^\Z,R\times \sigma)$};\
\node (3) at (0,-3) {$((\{0,1\}^\Z)^3,\sigma^{\times 3},\{(\nu_{\eta^*}\triangle\nu_\eta)\vee\kappa : \kappa\in \mathcal{P}(\{0,1\}^\Z)\})$};
\node (4) at (0,-4.5) {$([\underline{\varphi},\varphi],\sigma)$};
\node (5) at (0,-6) {$(X_{\varphi'},\sigma)$};

\draw [->] (2) -- node[auto] {$(\underline{\varphi}\otimes \varphi)\times id$} (3);
\draw [->] (3) -- node[auto] {$N$} (4);
\draw [->] (4) -- node[auto] {$id$} (5);

\draw (2) -- (5,-1.5);
\draw (5,-1.5) -- node[auto] {$M_H$} (5,-4.5);
\draw[->] (5,-4.5) -- (4);
\end{tikzpicture}
\]
Notice that if we prove that the above diagram is a commutative and surjective dynamical diagram then indeed we get:
\begin{itemize}
\item $\mathcal{P}(X_{\varphi'})=\mathcal{P}(X_{\eta'})=\mathcal{P}(X_\eta)=\mathcal{P}([\underline{\varphi},\varphi])$,
\item $\mathcal{P}(X_\eta)=\{N_\ast((\nu_{\eta^*}\triangle\nu_\eta)\vee\kappa):\kappa\in\mathcal{P}(\{0,1\}^\Z)\}=\{(M_H)_\ast(m_H\vee\kappa) :\kappa\in\mathcal{P}(\{0,1\}^\Z ) \}$.
\end{itemize} 

\paragraph{On Theorem~\ref{thme}} 
The main idea of the proof of Theorem~\ref{thme} is to equip the diagram
\[
(H\times \{0,1\}^\Z,R\times\sigma) \xrightarrow{M_H} (X_\eta,\sigma)
\]
(which is surjective by Theorem~\ref{glowne}) with an ``intermediate'' vertex:
\[
(H\times \{0,1\}^\Z,R\times\sigma) \xrightarrow{\Psi} (H\times\{0,1\}^\Z,\widetilde{R})\xrightarrow{\Phi} (X_\eta,\sigma),
\]
where $\widetilde{R}$ is a certain skew product over $R\colon H\to H$ and the maps $\Phi$ and $\Psi$ are morphisms defined later. We prove then that $h(H\times\{0,1\}^\Z,\widetilde{R})=\overline{d}-d^*$ (which equals to $h(X_\eta)$ by Theorem~\ref{entropia}) and prove the intrinsic ergodicity of $(H\times\{0,1\}^\Z,\widetilde{R})$. For the details, see Section~\ref{wewna}.

\section{Invariant measures}
\label{se:miary}

Before we begin working on the description of $\mathcal{P}(X_\eta)$, let us concentrate on $X_\eta$ itself. Keller~\cite{MR4280951} proved that for any taut set $\mathscr{B}$, the subshift $X_\eta$ is in a way ``hereditary''. We rephrase his result in the following way.
\begin{proposition}\label{between_etas}
For any $\mathscr{B}\subseteq \mathbb{N}$, we have
\(
X_\eta\subseteq\overline{[\eta^*,\eta]}\subseteq X_\varphi.
\)
In particular, if $\mathscr{B}$ is taut, $X_\eta=\overline{[\eta^*,\eta]}=X_\varphi$.
\end{proposition}
\begin{proof}
Clearly, $\varphi(\Delta(0))=\eta\in [\eta^*,\eta]$. Moreover, by Corollary 1.4 in \cite{MR4280951}, we have $\eta^*=\underline{\varphi}(\Delta(0))$. Thus, $[\eta^*,\eta]\subseteq[\underline{\varphi},\varphi]$. This yields
\[
X_\eta\subseteq \overline{[\eta^*,\eta]}\subseteq \overline{[\underline{\varphi},\varphi]} =X_\varphi.
\]
By Corollary~1.2 in~\cite{MR4280951}, if $\mathscr{B}$ is taut, we have $X_\eta=X_\varphi$, which completes the proof.
\end{proof}

\subsection{Proof of Theorem~\ref{pomocnicze} and Theorem~\ref{wnioC}}
\label{subsec:GenPoints}

\subsubsection{\texorpdfstring{\eqref{diagBC}}{(D\_B,C)} is a (commutative) dynamical diagram}

We will need a certain lemma from~\cite{MR4525753} about ``lifting'' quasi-generic points to joinings. We formulate it here for $\Z$-actions, while the original version is more general (the result is true for actions of countable cancellative semigroups and arbitrary Følner sequences).
\begin{theorem}[Theorem 5.16 in~\cite{MR4525753}]\label{downar}
Let $\mathcal{A}_1,\mathcal{A}_2$ be finite alphabets. If $x\in\mathcal{A}_1^\Z$ is quasi-generic for $\nu$ along \( (\ell_{i}) \) and $\nu\vee\kappa\in \mathcal{P}(\mathcal{A}_1^\Z\times \mathcal{A}_2^\Z,\sigma\times \sigma)$, then there exists $y\in\mathcal{A}_2^\Z$ such that the pair $(x,y)$ is quasi-generic for $\nu\vee\kappa$ along some subsequence of \( (\ell_{i}) \).
\end{theorem}
Let $(\ell_i)$ be a sequence realizing the lower density of $\mathcal{M}_\mathscr{B}$ and suppose that $X_{\eta^*}$ is uniquely ergodic (in particular, this happens if $\eta^*$ is a regular Toeplitz sequence). If the pair $(\eta^*,\eta)$ is quasi-generic along a subsequence $(\ell_{i_j})$ of $(\ell_i)$ for some measure then this limit measure must be a joining of $\nu_{\eta^*}$ and $\nu_\eta$. In fact, 
we have the following lemma which we will prove in a moment.
\begin{lemma}
\label{prop:PhiMHGeneric}
Let $\mathscr{B}\subseteq \mathbb{N}$ be such that $X_{\eta^*}$ is uniquely ergodic. Let $(\ell_i)$ be any sequence realizing the lower density of $\mathcal{M}_{\mathscr{B}}$. Then the point $(\eta^*,\eta)$ is quasi-generic along $(\ell_i)$ for $(\underline{\varphi}\otimes\varphi)_\ast(m_H)$.
% In particular,% $\nu_{\eta^*}\vee \nu_\eta=(\underline{\varphi}\otimes\varphi)_\ast(m_H)$ and thus
%\begin{equation*}
%\begin{tikzpicture}[baseline=(current  bounding  box.center)]
%\node (T) at (0,0) {$(H\times \{0,1\}^\Z,R\times \sigma)$};
%\node (M) at (0,-1.5) {$ ((\{0,1\}^{\Z})^3,\sigma^{\times 3},\{\nu_{\eta^*}\vee\nu_\eta\vee\kappa:\kappa\in\mathcal{P}(\{0,1\}^\Z)\})$};
%\draw[->]  (T) edge node[auto] {$(\underline{\varphi}\otimes \varphi)\times id$} (M);
%\end{tikzpicture}
%\end{equation*}
%is commutative.
\end{lemma}

\begin{remark}\label{ttt}
Instead of $(\underline{\varphi}\otimes\varphi)_\ast(m_H)$ we will usually write $\nu_{\eta^*}\triangle \nu_{\eta}$.
In this subsection we will only use that $(\eta^*,\eta)$ is quasi-generic along $(\ell_i)$, while the specific form of the limit measure will be used later. Let us justify here our notation $\nu_{\eta^*}\triangle \nu_\eta$ and show that this is a certain off-diagonal joining with marginals $\nu_{\eta^*}$ and $\nu_\eta$. Indeed, by~\eqref{eqn:PhiGamma}, we have
\[
\nu_{\eta^*}\triangle\nu_\eta=((\varphi^*\circ \Gamma_{H,H^*})\otimes \varphi)_\ast(m_H).
\]
%where $\varphi^*$ is defined for $\mathscr{B}^*$ in the same way as $\varphi$ for $\mathscr{B}$, $H^*=\overline{\{(n,n,n,\dots):n\in\mathbb{Z}\}}\subseteq \prod_{b^*\in\mathscr{B}^*}\Z/b^*\Z$ and $\Gamma_{H,H^*}\colon H\to H^*$ is the unique continuous group homomorphism that sends $(1,1,1,\dots)\in H$ to $(1,1,1,\dots)\in H^*$. 
Notice that
\[
S:=\varphi^*\circ \Gamma_{H,H^*}\circ \varphi^{-1} \colon (\{0,1\}^\Z,\sigma,\nu_\eta)\to(\{0,1\}^\Z,\sigma,\nu_{\eta^*})
\]
is a factoring map. Moreover, for any measurable sets $A,B\subseteq \{0,1\}^\Z$, we have
\begin{multline*}
\triangle_S(A\times B)=\nu_\eta((\varphi^*\circ \Gamma_{H,H^*}\circ \varphi^{-1})^{-1}(A)\cap B)\\
=m_H(\varphi^{-1}(B)\cap (\varphi^*\circ \Gamma_{H,H^*})^{-1}(A))=((\varphi^*\circ \Gamma_{H,H^*})\otimes \varphi)_\ast(m_H)(A\times B)=\nu_{\eta^*}\triangle \nu_\eta(A\times B).
\end{multline*}
%\textcolor{blue}{please check: need \( A \) and  \( B \) be swapped in \(m_H(\varphi^{-1}(A)\cap (\varphi^*\circ \Gamma_{H,H^*})^{-1}(B))\) ? I think that it was correct the first time around}
\end{remark}
Recall also that it was shown in~\cite{MR3803141} that
\begin{equation}\label{jeszczeto}
\eta \text{ and }\eta' \text{ differ along }(\ell_i)_{i\geq 1} \text{ on a subset of zero density},
\end{equation}
where $(\ell_i)_{i\geq 1}$ is any sequence realizing the lower density of $\mathcal{M}_{\mathscr{B}'}$. Moreover (see the proof of Lemma~4.11 in~\cite{MR3803141}) any sequence $(\ell_i)$ realizing the lower density of $\mathcal{M}_{\mathscr{B}'}$ is also realizing the lower density of $\mathcal{M}_\mathscr{B}$. For any such $(\ell_i)_{i\geq 1}$, also
\begin{equation}\label{dlatej}
(\eta^*,\eta') \text{ is quasi-generic for }\nu_{\eta^*}\triangle\nu_\eta\text{ along }(\ell_i),
\end{equation}
whenever $X_{\eta^*}$ is uniquely ergodic.

Let us now prove that~\eqref{diagBC} is indeed a dynamical diagram (and since it is linear, it is then commutative by Remark~\ref{rem:linDiagComm}). Notice that it suffices to show that for any measure of the form $N_\ast(\rho)$, where $\rho=(\nu_{\eta^*}\triangle \nu_\eta)\vee \kappa$ with $\kappa\in\mathcal{P}(\{0,1\}^\Z)$, we have $N_\ast(\rho)(X_{\varphi'})=1$. In order to see that this is indeed the case, fix such a measure $\rho$. It follows by Theorem~\ref{downar} and by~\eqref{dlatej} that $\rho$ has a quasi-generic point of the form $(\eta^*,\eta',y)$ with $y\in \{0,1\}^\Z$. Therefore, $z:=N(\eta^*,\eta',y)$ is quasi-generic for $N_\ast(\rho)$ and thus $N_\ast(\rho)(X_z)=1$. It remains to notice that $\eta^*\leq z\leq \eta'$. Thus,
\[
X_z \subseteq \overline{[\eta^{*},\eta']}= X_{\varphi'},
\]
where the last equality follows from Proposition~\ref{between_etas}.

\begin{proof}[Proof of Lemma~\ref{prop:PhiMHGeneric}]
Fix $(\ell_i)$ which realizes the lower density of $\mathcal{M}_{\mathscr{B}}$. By a pure measure theory argument (see the proof of Theorem 4.1 in~\cite{MR3428961}), we only
need to prove that 
\begin{equation*}
\frac{1}{\ell_i} \sum_{n\leq \ell_i} \delta_{ ( \sigma^n\underline{\varphi}(\Delta(0)),\sigma^n\varphi(\Delta(0)))}(\underline{A}\times A) \to (\underline{\varphi}\otimes \varphi)_\ast(m_H)(\underline{A}\times A)
\end{equation*}
for 
\begin{equation*}
\underline{A}=\{x\in \{0,1\}^\Z : x|_{\underline{S}} \equiv 0\} \text{ and }A=\{x\in \{0,1\}^\Z : x|_{S} \equiv 0\},
\end{equation*}
with $\underline{S},S\subseteq \Z$ being arbitrary finite sets. By $\sigma\circ \underline{\varphi}=\underline{\varphi}\circ R$ and $\sigma\circ \varphi=\varphi\circ R$, this is equivalent to proving that
\begin{equation*} 
\lim_{i\to\infty}\frac{1}{\ell_i}\sum_{n\leq \ell_i}\mathbf{1}_{\underline{\varphi}^{-1}(\underline{A})\times \varphi^{-1}(A)}(R^n (\Delta(0)),R^n (\Delta(0)))=(\underline{\varphi}\otimes \varphi)_\ast(m_H)(\underline{A}\times A).
\end{equation*}
The main underlying idea is to approximate $\underline{\varphi}^{-1}(\underline{A})$ and $\varphi^{-1}(A)$ by clopen sets, so that we can use the ergodicity properties of rotations. We will begin with the right-hand-side, as it is easier (the approximation of the left-hand-side requires the use of the Davenport-Erd\H{o}s theorem, i.e.~\eqref{DE}).

\paragraph{Approximation of the right-hand-side}
We have
\[
C:=\varphi^{-1}(A)=\bigcap_{s\in S}R^{-s}W^c.
\]
Let, for $K\geq 1$,
\[
W_K:=\{h\in H : h_b\neq 0 \text{ for all }b\in\mathscr{B}_K\}.
\]
Let
\begin{equation}\label{ceka}
C_K:=\bigcap_{s\in S}R^{-s}W_K^c .
\end{equation}
Each $W_K$ is clopen and $W_K \searrow W$ when $K\to\infty$. Thus, given $\varepsilon>0$, for $K$ large enough, we have
\begin{equation}\label{ny1}
m_H(C\triangle C_K)<\varepsilon.
\end{equation}

Recall from (\ref{eqn:PhiGamma}) that $\underline{\varphi}=\varphi^*\circ \Gamma_{H,H^*}$ and let
\[
\underline{C}:=\underline{\varphi}^{-1}(\underline{A})=\Gamma_{H,H^*}^{-1}((\varphi^*)^{-1}(\underline{A}))=\Gamma_{H,H^*}^{-1}\bigcap_{s\in\underline{S}}(R^*)^{-s}(W^*)^c=\bigcap_{s\in\underline{S}}\Gamma_{H,H^*}^{-1}(R^*)^{-s}(W^*)^c.
\]
Define $W_K^*$ in a similar way as $W_K$:
\[
W^*_K=\{h^*\in H^* : h^*_{b^*}\not\equiv 0 \text{ for all }b^*\in\mathscr{B}^*_K\}.
\]
Finally, let
\begin{equation}\label{ceka1}
\underline{C}_K:=\bigcap_{s\in\underline{S}}\Gamma_{H,H^*}^{-1}(R^*)^{-s}(W_{K}^*)^c.
\end{equation}
Then, for $K$ large enough,
\begin{equation}\label{ny2}
m_H(\underline{C}\triangle \underline{C}_K)<\varepsilon.
\end{equation}

Notice that
\[
(\underline{\varphi}\otimes \varphi)_\ast(m_H)(\underline{A}\times A)=m_H(\underline{\varphi}^{-1}(\underline{A})\cap \varphi^{-1}(A))=m_H(\underline{C}\cap C). 
\]
Thus, it follows by~\eqref{ny1} and~\eqref{ny2} that
\[
\left|(\underline{\varphi}\otimes \varphi)_\ast(m_H)(\underline{A}\times A) - m_H(\underline{C}_K\cap C_K) \right|\leq 2\varepsilon
\]
for $K$ sufficiently large.

\paragraph{Approximation of the left-hand-side}

Let $(\ell_i)_{i\geq 1}$ be a sequence realizing the lower density of $\mathcal{M}_{\mathscr{B}}$. By definition
\[
\frac{1}{\ell_i}\sum_{n\leq \ell_i}\mathbf{1}_{\underline{A}\times A}(\underline{\varphi}(R^n (\Delta(0))),\varphi(R^n (\Delta(0))))=\frac{1}{\ell_i}\sum_{n\leq \ell_i}\mathbf{1}_{\underline{C}}(R^n(\Delta(0)))\mathbf{1}_{C}(R^n(\Delta(0))).
\]
Moreover, for $K$ large enough
\begin{align*}
\begin{split}
\lim_{i\to\infty}&\left|\frac{1}{\ell_i}\sum_{n\leq \ell_i}\mathbf{1}_{\underline{C}}(R^n(\Delta(0)))\mathbf{1}_{C}(R^n(\Delta(0)))-\frac{1}{\ell_i}\sum_{n\leq \ell_i}\mathbf{1}_{\underline{C}}(R^n(\Delta(0)))\mathbf{1}_{C_K}(R^n(\Delta(0)))\right|\\
&\leq \lim_{i\to\infty}\frac{1}{\ell_i}\sum_{n\leq \ell_i}\mathbf{1}_{C\setminus C_K}(R^n(\Delta(0)))
\leq
\lim_{i\to\infty}\sum_{s\in S}\frac{1}{\ell_i}\sum_{n\leq \ell_i}\mathbf{1}_{W^c\setminus W_K^c}(R^{n+s}(\Delta(0))\\
&=|S| \cdot \lim_{i\to\infty} \frac{1}{\ell_i}\sum_{n\leq \ell_i}\mathbf{1}_{W^c\setminus W_K^c}(R^{n} (\Delta(0))) = |S| \cdot \lim_{i\to\infty} \frac{1}{\ell_{i}} |[1,\ell_i]\cap (\mathcal{M}_{\mathscr{B}}\setminus \mathcal{M}_{\mathscr{B}_K})|<\varepsilon,
\end{split}
\end{align*}
where the second inequality follows from
\[
C\setminus C_K \subseteq \bigcup_{s\in S}R^{-s}(W^c\setminus W_K^c),
\]
the last equality from
\[
R^n(\Delta(0)) \in W^c\setminus W_K^c \iff n\in \mathcal{M}_{\mathscr{B}}\setminus \mathcal{M}_{\mathscr{B}_K}
\]
and the last inequality is a consequence of the Davenport-Erd\H{o}s theorem (i.e.~\eqref{DE}) -- we use that $(\ell_i)_{i\geq 1}$ is a specific sequence and that $K$ is large only for this last inequality.

We will now use similar arguments for \( C_{K} \) and \( \underline{C}_{K} \) instead of \(  C \) and \( \underline{C} \). We have
\begin{align*}
\lim_{i\to\infty}&\left|\frac{1}{\ell_i}\sum_{n\leq \ell_i} \mathbf{1}_{\underline{C}}(R^n(\Delta(0)))\mathbf{1}_{C_K}(R^n(\Delta(0)))- \frac{1}{\ell_i}\sum_{n\leq \ell_i}\mathbf{1}_{\underline{C}_K}(R^n(\Delta(0)))\mathbf{1}_{C_K}(R^n(\Delta(0))) \right|\\
&\leq \lim_{k\to\infty}\frac{1}{\ell_i}\sum_{n\leq \ell_i}\mathbf{1}_{\underline{C}\setminus \underline{C}_K}(R^n(\Delta(0)))\\
&\leq |\underline{S}|\cdot \lim_{i\to \infty}\frac{1}{\ell_i}\sum_{n\leq \ell_i}\mathbf{1}_{\Gamma_{H,H^*}^{-1}((W^*)^c\setminus (W_K^*)^c)}(R^n(\Delta(0)))\\
&=|\underline{S}|\cdot \lim_{i\to\infty}\frac{1}{\ell_i}\sum_{n\leq \ell_i}\mathbf{1}_{(W^*)^c\setminus (W_K^*)^c}((R^*)^n(\Delta^*(0)))\\
&=|\underline{S}|\cdot d(\mathcal{M}_{\mathscr{B}^*}\setminus \mathcal{M}_{\mathscr{B}^*_K})<\varepsilon,
\end{align*}
where in the second inequality we used
\[
\underline{C}\setminus \underline{C}_K \subseteq \bigcup_{s\in\underline{S}} R^{-s}\Gamma_{H,H^*}^{-1}((W^*)^c\setminus (W_K^*)^c),
\]
the first equality follows from $\Gamma_{H,H^*}(R^n\Delta(0))=(R^*)^n \Gamma_{H,H^*}(\Delta(0))=(R^*)^n(\Delta^*(0))$, the second equality is a consequence of
\[
(R^*)^n(\Delta^*(0))\in (W^*)^c\setminus (W^*_K)^c\iff n\in \mathcal{M}_{\mathscr{B}^*}\setminus \mathcal{M}_{\mathscr{B}^*_K}
\]
and the last inequality follows by~\eqref{DE}, i.e.~the Davenport-Erd\H{o}s theorem (notice that we use here that $X_{\eta^*}$ is uniquely ergodic, so, in particular, $\mathscr{B}^*$ is Besicovitch and thus the density of $\mathcal{M}_{\mathscr{B}^*}$ along $(\ell_i)$ is just its natural density).

\paragraph{Convergence for clopen sets}
After the above reductions, it remains to prove that
\[
\lim_{i\to\infty}\frac{1}{\ell_i}\sum_{n\leq \ell_i}\mathbf{1}_{\underline{C}_K}(R^n(\Delta(0)))\mathbf{1}_{C_K}(R^n(\Delta(0)))=m_H(\underline{C}_K\cap C_K).
\]
However, both $C_K$ and $\underline{C}_K$ are clopen (recall~\eqref{ceka} and~\eqref{ceka1}) and thus the claim follows directly by the unique ergodicity of $R$.
\end{proof}

\subsubsection{\texorpdfstring{\eqref{diagBC}}{(D\_B,C)} is surjective}
This part of the proof relies mostly on certain natural periodic approximations of $\eta$ and $\eta^*$. More precisely, we will need a periodic approximation of $\eta$ from above and of $\eta^*$ from below.

For each $K\geq 1$, we set $\sB_K:=\{b\in\mathscr{B} : b\leq K\}$ and $\sB_K^*=\{b^*\in\sB^* : b^*\leq K\}$. We define $\varphi_K\colon H \to \{0,1\}^\Z$ by
\[
\varphi_K(h)(n)=1 \iff (R^{n} h)_{b} \neq 0\text{ for all }b\in\sB_K.
\]
Recall that there is a continuous group homomorphism \( \Gamma_{H,H^{*}}\colon H \to H^{*} \) with $\underline{\varphi}(h)=\varphi^*(\Gamma_{H,H^*}(h))$, see (\ref{eqn:PhiGamma}). We define $\underline{\varphi}_K\colon H\to \{0,1\}^\Z$ by
\begin{align*}
\underline{\varphi}_K(h)(n)&=\begin{cases}
\underline{\varphi}(h)(n)& \text{if }n\in \text{Per}(\underline{\varphi}(h),\lcm(\sB_K^*)),\\
0& \text{otherwise},\end{cases}\\
&= \begin{cases}
1 & \text{if }n\in \text{Per}(\varphi^*(\Gamma_{H,H^*}(h)), 1 , \lcm(\sB_K^*)),\\
0 & \text{otherwise}. \end{cases}
\end{align*}
Note that for every $h\in H$, we have 
\begin{equation}\label{order}
\underline{\varphi}_{K}(h) \leq \underline{\varphi}(h)\leq \varphi(h) \leq \varphi_{K}(h).
\end{equation}

\begin{lemma}\label{skonczenie}
For any $K\geq 1$, functions $\varphi_K$ and $\underline{\varphi}_K$ depend on a finite number of coordinates. In particular, they are continuous.
\end{lemma}

\begin{proof}
For $\varphi_K$ the assertion is clear by the very definition. Let us now turn to $\underline{\varphi}_K$. To shorten notation, we will write \( h^*=\Gamma_{H,H^*}(h) \) and $s^{*}=\lcm(\sB_K^*)$. We will show that \( \text{Per}(\varphi^*(h^{*}_{1}), 1 , s^{*}) =  \text{Per}(\varphi^*(h^{*}_{2}), 1 , s^{*}) \) whenever \( h^{*}_{1} \) and \( h^{*}_{2} \) agree on \( \sB_{K}^{*} \). Since there exists \( L \in \N \) such that every \( b^{*} \in \sB^{*}_{K} \) divides some \( b \in \sB_{L} \), by  Lemma~\ref{W3} it then follows that $\underline{\varphi}_{K}( h )$ is determined by \( ( h_{b} )_{b \in B_{L}} \). To see that \( \text{Per}(\varphi^*(h^{*}), 1 , s^{*}) \) depends only on \( \sB_{K}^{*} \), take \( h_{1}^{*} , h_{2}^{*} \in H^{*} = \overline{\Delta^{*}( \Z )} \) with \( ( h_{2}^{*} - h_{1}^{*} )_{b^{*}} = 0 \) for all \( b^{*} \in \sB_{K}^{*} \). Then there exists a sequence \( (n_{k}) \) with \( \Delta^{*}( n_{k} ) \to  h_{2}^{*} - h_{1}^{*} \) and \( \lcm( \sB^{*}_{K} ) = s^{*} \mid n_{k} \). We notice that
\[ \text{Per}(\varphi^*( h^{*}_{1} + \Delta^{*}( n_{k} ) ), 1 , s^{*}) = \text{Per}( \sigma^{n_{k}} \varphi^*( h^{*}_{1} ), 1 , s^{*}) = \text{Per}(\varphi^*(h^{*}_{1} ), 1 , s^{*}) - n_{k} = \text{Per}(\varphi^*( h^{*}_{1} ), 1 , s^{*}) , \]
since \( \text{Per}(\varphi^*( h^{*}_{1} ), 1 , s^{*}) \) is an \( s^{*} \)-periodic set. In particular, for every \( j \in \text{Per}(\varphi^*( h^{*}_{1} ), 1 , s^{*}) \) we get \( \varphi^*( h^{*}_{1} + \Delta^{*}( n_{k} ) )( j ) = 1 \) for all \( k \). Since \( h^{*}_{1} + \Delta^{*}( n_{k} ) \) converges to \( h_{2}^{*} \), and \( \varphi^{*} \) is coordinatewise upper semicontinuous, this yields \( \varphi^*( h^{*}_{2} )( j ) = 1 \) for all \( j \in \text{Per}(\varphi^*( h^{*}_{1} ), 1 , s^{*}) \) and hence \( \text{Per}(\varphi^*( h^{*}_{1} ), 1 , s^{*}) \subseteq \text{Per}(\varphi^*( h^{*}_{2} ), 1 , s^{*}) \). By the symmetry between \( h^{*}_{1} \) and \( h^{*}_{2} \), the converse inclusion follows from the same argument, thus proving the claim. 
\end{proof}

Similar to \( [\underline{\varphi} ,\varphi ] \), we define \( [\underline{\varphi}_K,\varphi_K] :=\{x\in \{0,1\}^Z : \underline{\varphi}_K(h)\leq x\leq\varphi_K(h)\text{ for some }h\in H\} \).

\begin{lemma}\label{propo9}
The set \( [\underline{\varphi}_K,\varphi_K] \subseteq \{ 0 , 1 \}^{\Z} \) is a subshift.
\end{lemma}

\begin{proof}
That \( [\underline{\varphi}_K,\varphi_K] \) is closed follows immediately from the continuity of $\varphi_K$ and $\underline{\varphi}_K$. In addition, it is $\sigma$-invariant as 
\[
\varphi_K\circ R=\sigma \circ \varphi_K \text{ and }\underline{\varphi}_K\circ R=\sigma \circ \underline{\varphi}_K.
\]
Indeed, the first equality holds as $\varphi_K$ is a coding of orbits of points in $H$ with respect to $\{h\in H : h_{b}\neq 0 \text{ for all }b\in\sB_K\}$. The second equality follows from the definition of \( \underline{\varphi}_{K} \) in terms of \(\underline{\varphi}\), the equality $\underline{\varphi}\circ R=\sigma\circ \underline{\varphi}$ (since $\underline{\varphi}$ is a coding) and $\Per(\sigma x,s)=\Per(x,s)-1$ for $x$ in the orbit closure of a Toeplitz sequence.
\end{proof}

We set $ \eta_K := \varphi_{K}(\Delta(0))=\mathbf{1}_{\mathcal{F}_{\mathscr{B}_K}}$ and \( \underline{\eta}_K := \underline{\varphi}_K(\Delta(0)) \). Then \(\underline{\eta}_K\leq \eta^*\leq \eta\leq \eta_K\) (this is a special case of~\eqref{order} for $h = \Delta(0)$).
\begin{lemma}\label{lm13}
Let $\mathscr{B}\subseteq \N$ and suppose that $\eta^*$ is a regular Toeplitz sequence. Let $(\ell_{i})$ be a sequence realizing the lower density of $\mathcal{M}_{\mathscr{B}}$. Then 
\[
\lim_{K\to\infty}\overline{d}_{(\ell_{i})}(\{n\in \N : (\underline{\eta}_{K}(n),\eta_{K}(n))\neq (\eta^*(n),\eta(n))\})=0.
\]
\end{lemma}

\begin{proof}
It suffices to notice that
\[
\lim_{K\to\infty}\overline{d}_{(\ell_i)}(\{n\in \N : \eta_{K}(n)\neq \eta(n)\})=0
\]
and
\[\lim_{K\to \infty}\overline{d}_{(\ell_i)}(\{n\in \N : \underline{\eta}_{K}(n)\neq \eta^*(n)\})=0.\]
The first assertion follows by the Davenport-Erd\H{o}s theorem (i.e.\ by~\eqref{DE}). For the second notice that \(\underline{\eta}_K(n)\neq \eta^*(n)\) implies that $n\not\in \operatorname{Per}(\eta^*,\lcm(\mathscr{B}_K^*))$. Thus,
\begin{equation}\label{nowe}
\lim_{K\to \infty}\overline{d}(\{n\in \N : \underline{\eta}_{K}(n)\neq \eta^*(n)\})\leq \lim_{K\to\infty}{d}\left(\Z\setminus \operatorname{Per}(\eta^*,\lcm(\mathscr{B}_K^*))\right)=0
\end{equation}
as $\eta^*$ is a regular Toeplitz sequence, cf.~\eqref{pozycje}.
\end{proof}

We will also need the following well-known fact related to quasi-generic points and the corresponding invariant measures (we skip its proof and refer the reader e.g.~to Appendix~C in~\cite{doktorat}, see also~\cite{MR3669782}).
\begin{proposition}\label{pr12}
Let $\mathcal{A}$ be a finite alphabet and suppose that $(\ell_{i}) \subseteq \N$ is an increasing sequence and that $x_K\in \mathcal{A}^\Z$ for $K\geq 1$ and $x\in  \mathcal{A}^\Z$ are such that 
\[
\lim_{K\to \infty} \overline{d}_{(\ell_{i})}(\{n\in\N : x_K(n)\neq x(n)\})=0.
\]
Suppose additionally that $x_K$, $K\geq 1$, and $x$ are quasi-generic along $(\ell_{i})$ for measures $\nu_K$, $K\geq 1$ and $\nu$, respectively. Then $\nu_K \to \nu$ in the weak topology.
\end{proposition}

Last but not least, we will need the following result in order to pass from the description of ergodic measures to that of all invariant measures on $X_\eta$.
\begin{proposition}\label{arsenin}
Suppose that for a subshift $X\subseteq \{0,1\}^\Z$, we have
\(
\mathcal{P}^e(X)\subseteq \{N_\ast((\nu_{\eta^*}\triangle \nu_\eta)\vee\kappa) : \kappa\in\mathcal{P}(\{0,1\}^\Z)\}.
\)
Then
\(
\mathcal{P}(X)\subseteq \{N_\ast((\nu_{\eta^*}\triangle \nu_\eta)\vee\kappa) : \kappa\in\mathcal{P}(\{0,1\}^\Z)\}.
\)
\end{proposition}
We skip the proof -- it is a repetition (with obvious changes such as replacing the map $M$ by $N$ and the Mirsky measure $\nu_\eta$ by the joining $\nu_{\eta^*}\triangle\nu_\eta$) of the proof of an analogous part of Theorem~4.1.23 in~\cite{doktorat} (more specifically, see page 66 therein). The main tool there is the ergodic decomposition and the Arsenin-Kunungui theorem on measurable selection (see, e.g., Theorem 18.18 in~\cite{MR1321597}). A more general result (with a detailed proof) will be published in~\cite{KLR}.

Now, fix $\nu\in\mathcal{P}^{e}(\overline{[\underline{\varphi},\varphi]}) = \mathcal{P}^e([\underline{\varphi},\varphi])$, see \eqref{etaregular}. Since $\nu$ is ergodic, since $[ \underline{\varphi} , \varphi ]\subseteq [\underline{\varphi}_K,\varphi_K]$ and the latter set is a subshift by Lemma~\ref{propo9}, there exists a generic point $u_K\in [\underline{\varphi}_K,\varphi_K]$ for $\nu$. Without loss of generality, we can assume
\[
\underline{\eta}_K\leq u_K\leq \eta_K.
\]
(Indeed, since $u_K\in [\underline{\varphi}_K,\varphi_K]$ there exists \( h \in H \) with \( \underline{\varphi}_K(h)\leq u_K\leq \varphi_K(h) \). If $j\in \N$ is such that $h_{b}+j=0  \bmod \ b$  for all $b$ in a sufficiently large, finite subset of $\sB$, then Lemma~\ref{skonczenie} shows \(
\underline{\varphi}_K(\Delta(0)) \leq \sigma^ju_K\leq \varphi_K( \Delta(0) ) \), where $\sigma^ju_K$ is generic for $\nu$.) Thus, there exits \( y_{K} \in \{ 0,1 \}^{\Z} \) such that \( u_K=N(\underline{\eta}_K,\eta_K,y_K) \). Notice that \( (\underline{\eta}_K,\eta_K,y_K) \) is quasi-generic for some measure \( \rho_K \). Using the periodicity of \( \underline{\eta}_{K} \) and \( \eta_K \), we hence obtain that $(\underline{\eta}_K,\eta_K)$ is generic for $(\pi_{1,2})_\ast(\rho_K)$. In addition, \( ( \eta^{*} , \eta ) \) is quasi-generic along \( (\ell_{i}) \) for \( \nu_{\eta^*}\triangle \nu_\eta\) by \eqref{jeszczeto} and \eqref{dlatej}. Thus, Proposition~\ref{pr12} and Lemma~\ref{lm13} yield
\[
(\pi_{1,2})_{\ast}(\rho_K) \to \nu_{\eta^*}\triangle \nu_\eta .
\]
If \( \rho \) is a limit of \( \rho_{K} \), it follows that $(\pi_{1,2})_\ast(\rho) =  \nu_{\eta^*}\triangle \nu_\eta $, so \( \rho \) is of the form \( \rho = (\nu_{\eta^*}\triangle \nu_\eta) \vee \kappa \) for some \( \kappa \in \mathcal{P}( \{ 0 , 1 \}^{\Z} ) \). Finally, since \( (\underline{\eta}_K,\eta_K,y_K) \) is quasi-generic for \( \rho_K \), it follows that \( u_K=N(\underline{\eta}_K,\eta_K,y_K) \) is quasi-generic for \( N_{*} (\rho_K) \). However, by assumption, \( u_{K} \) is also generic for \( \nu \), which yields \(\nu = N_{*} (\rho_K) \) for all \( K \in \N \), and thus \(\nu = N_{*}(\rho) \in \{ N_{\ast}( (\nu_{\eta^{*}} \triangle \nu_{\eta}) \vee \kappa ) : \kappa \in \mathcal{P}( \{ 0 , 1 \}^{\Z} ) \} \). This proves \( \mathcal{P}^e([\underline{\varphi},\varphi]) \subseteq \{ N_{\ast}( (\nu_{\eta^{*}} \triangle \nu_{\eta}) \vee \kappa ) : \kappa \in \mathcal{P}( \{ 0 , 1 \}^{\Z} ) \} \). To complete the proof of the surjectivity of $\mathbf{D}_{B,C}$, we use Proposition~\ref{arsenin}.

\subsection{Proof of Theorem~\ref{glowne}}\label{profgl}

\subsubsection{\texorpdfstring{\eqref{zato2}}{(D\_A)} is a commutative dynamical diagram}\label{prep}

The proof that~\eqref{zato2} is a commutative dynamical diagram uses two ingredients. The first of them is that the following is a dynamical diagram:
\begin{equation*}
\begin{tikzpicture}[baseline=(current  bounding  box.center)]
\node (T) at (0,0) {$(H\times \{0,1\}^\Z,R\times \sigma)$};
\node (M) at (0,-1.5) {$ ((\{0,1\}^{\Z})^3,\sigma^{\times 3},\{(\nu_{\eta^*}\triangle\nu_\eta)\vee\kappa:\kappa\in\mathcal{P}(\{0,1\}^\Z)\})$};
\draw[->]  (T) edge node[auto] {$(\underline{\varphi}\otimes \varphi)\times id$} (M);
\end{tikzpicture}
\end{equation*}
which is a consequence of Lemma~\ref{prop:PhiMHGeneric}.

The second ingredient that we need to prove the commutativity of~\eqref{zato2} is the following equality (that holds everywhere):
\[%\label{p}
N\circ ((\underline{\varphi}\otimes \varphi)\times id)=M_H,
\]
which can be checked in a direct one-line calculation.
%(Then is suffices to use the commutativity of $N$ to obtain the commutativity of whole~\eqref{zato2}.)
%Before we provide the proof of Proposition~\ref{prop:PhiMHGeneric}, let us see first how to complete the proof of the commutativity of $\mathbf{D}_A$. Let $\nu:=(M_H)_\ast(m_H\vee\kappa)$ for some measure $\kappa\in\mathcal{P}(\{0,1\}^\Z)$. It follows by~\eqref{p} that
%\[
%\nu=N_\ast ((\underline{\varphi}\otimes \varphi)\times id)_\ast (m_H\vee\kappa)=N_\ast ((\underline{\varphi}\otimes \varphi)_\ast (m_H)\vee\kappa).
%\]
%(Note that the second equality here means that there exists a joining of $(\underline{\varphi}\otimes \varphi)_\ast(m_H)$ and $\kappa$ such that the above holds.) Therefore, by Proposition~\ref{prop:PhiMHGeneric}, we obtain
%\[
%\nu=N_\ast(\nu_{\eta^*}\triangle\nu_\eta\vee\kappa)
%\]
%and it follows that (A) indeed holds.

\subsubsection{\texorpdfstring{\eqref{zato2}}{(D\_A)} is surjective}\label{surje}

Recall that by~\eqref{a} and by the surjectivity of~\eqref{diagBC}, the diagram
\begin{equation*}
\begin{tikzpicture}[baseline=(current  bounding  box.center)]
\node (M) at (0,-1.5) {$ ((\{0,1\}^{\Z})^3,\sigma^{\times 3},\{(\nu_{\eta^*}\triangle\nu_\eta)\vee\kappa:\kappa\in\mathcal{P}(\{0,1\}^\Z)\})$};
\node (B) at (0,-3) {$(X_\eta,\sigma)$};
\draw[->]  (M) edge node[auto] {$N$} (B);
\end{tikzpicture}
\end{equation*}
is surjective. Thus (using also the commutativity of~\eqref{zato2}), in order to prove the surjectivity of~\eqref{zato2}, it suffices
%To prove the surjectivity of~\eqref{zato2}, in view of the surjectivity of 
%\begin{equation*}
%\begin{tikzpicture}[baseline=(current  bounding  box.center)]
%\node (M) at (0,-1.5) {$ ((\{0,1\}^{\Z})^3,\sigma^{\times 3},\{(\nu_{\eta^*}\triangle\nu_\eta)\vee\kappa:\kappa\in\mathcal{P}(\{0,1\}^\Z)\})$};
%\node (B) at (0,-3) {$(X_\eta,\sigma)$};
%\draw[->]  (M) edge node[auto] {$N$} (B);
%\end{tikzpicture}
%\end{equation*}
%\textcolor{red}{and the commutativity of~\eqref{zato2} and the surjectivity of~\eqref{diagBC}, by~\eqref{a}}, it suffices 
to prove that
\begin{equation}\label{ten}
\begin{tikzpicture}[baseline=(current  bounding  box.center)]
\node (T) at (0,0) {$(H\times \{0,1\}^\Z,R\times \sigma)$};
\node (M) at (0,-1.5) {$ ((\{0,1\}^{\Z})^3,\sigma^{\times 3},\{(\nu_{\eta^*}\triangle\nu_\eta)\vee\kappa:\kappa\in\mathcal{P}(\{0,1\}^\Z)\})$};
\draw[->]  (T) edge node[auto] {$(\underline{\varphi}\otimes \varphi)\times id$} (M);
\end{tikzpicture}
\end{equation}
is surjective. However, by Lemma~\ref{prop:PhiMHGeneric} (cf.\ Remark~\ref{ttt}), $((\{0,1\}^\Z)^2,\sigma^{\times 2},\nu_{\eta^*}\triangle\nu_\eta)$ is a factor of $(H,R,m_H)$ via $\underline{\varphi}\otimes \varphi$, so given any joining $(\nu_{\eta^*}\triangle\nu_\eta)\vee\kappa$, it suffices to take its relatively independent extension to a joining of $m_H$ with $\kappa$ to conclude that~\eqref{ten} is surjective.

\section{Tautness and combinatorics}

\begin{proof}[Proof of Proposition~\ref{NNN}]
We first show that the conditions (a')-(f') are all equivalent. We then pass to proving that, in fact, they are also equivalent to each of (a)-(f).

Note that the implications (d') \( \implies \) (c') \( \implies \) (f') \( \implies \) (e') are immediate. Next we show (e') \( \implies \) (a'). It was shown in~\cite{MR3920387} that for taut sets, the corresponding Mirsky measure is of full support in the corresponding $\mathscr{B}$-free subshift. Applying this to $\mathscr{C}$, we conclude that each block that appears on $X_{\eta_\mathscr{C}}$ is of positive $\nu_{\eta_\mathscr{C}}$-measure. Thus it follows from (e') that $X_{\eta_\mathscr{C}}\subseteq X_{\eta'}$, and hence $\widetilde{X}_{\eta_\mathscr{C}} \subseteq \widetilde{X}_{\eta'}$. Using~\eqref{AA}, we obtain that $ \forall_{b'\in\mathscr{B}'} \exists_{c\in\mathscr{C}}$\ $c\divides b'$. Moreover $X_{\eta_\mathscr{C}}\subseteq X_\eta$ implies $X_{\eta^*}\subseteq X_{\eta_\mathscr{C}}$, since $X_{\eta^*}$ is the unique minimal subset of $X_\eta$. This yields $\widetilde{X}_{\eta^*}\subseteq \widetilde{X}_{\eta_\mathscr{C}}$. Using~\eqref{AA} again, we obtain that $\forall_{c\in\mathscr{C}}\exists_{b^*\in\mathscr{B}^*}\ b^*\divides c$, which proves (a'). Next we note that (a') \( \implies \) (b') by the very definition of $\eta',\eta_\mathscr{C}$ and $\eta^*$. To finish the first part, it only remains to notice that by Proposition~\ref{between_etas} and tautness of \( \mathscr{B}^{\prime} \) it follows that $X_{\eta^{\prime}}=\overline{ [\eta^*,\eta'] } $, which yields (b') \( \implies \) (d').

Since the proof of (b') \( \implies \) (d') was the only place where we used the tautness of \( \mathscr{B}^{\prime} \), the same arguments as above show also that (d) \( \implies \) (c) \( \implies \) (f) \( \implies \) (e) \( \implies \) (a) \( \implies \) (b). We now prove (b) \( \implies \) (b'). As (b') \( \implies \) (d') was already shown, and as (d') \( \implies \) (d) follows directly from \( X_{ \eta^{\prime} } \subseteq X_{\eta} \), this will finish the proof. Thus, suppose that $\eta_\mathscr{C}\leq \eta$. It follows then by~\eqref{AA} that $\nu_{\eta_\mathscr{C}}\in\mathcal{P}(\widetilde{X}_\eta)=\mathcal{P}(\widetilde{X}_{\eta'})$. Applying again~\eqref{AA}, we obtain $\eta_\mathscr{C}\leq \eta'$, and hence (b').
\end{proof}

\section{Entropy}\label{se:entropia}

\subsection{Entropy of \texorpdfstring{$X_\eta$}{X\_eta}: proof of Theorem~\ref{entropia} and of Corollary~\ref{zerowanie}}\label{se:em}

\begin{remark}\label{entropy_unique_erg}
If $X_\eta$ is uniquely ergodic then the Mirsky measure $\nu_\eta$ (whose entropy is zero) is the unique invariant measure and it follows immediately by the variational principle that $h(X_\eta)=0$.
\end{remark}

\begin{proof}[Proof of Theorem~\ref{entropia}]
To show the inequality \( h(X_{\eta}) \geq \overline{d}-\overline{d}^*\), we first assume that \( \sB \) is taut and consider the following block: $B=\eta[1,n]\in\{0,1\}^n$. Then $B(\ell)=0$ for any $\ell\in\cM_\sB\cap[1,n]$ and $B(\ell)=1$ for any $\ell\in\mathcal{F}_{\sB}\cap[1,n]$ (so, in particular, for any $\ell\in\mathcal{F}_{\sB^*}\cap[1,n]$). It follows by Proposition~\ref{between_etas} that any block $C\in\{0,1\}^n$ that agrees with $B$ on the positions belonging to $\mathcal{M}_\mathscr{B}\cup\mathcal{F}_{\mathscr{B}^*}$ also appears in $X_\eta$. There are 
\[%\label{entropy_below}
2^{n-|(\cM_\sB\cup\mathcal{F}_{\sB^*})\cap[1,n]|}=2^{|\mathcal{F}_\sB\cap[1,n]|-|\mathcal{F}_{\sB^*}\cap[1,n]|} \]
such blocks $C$ (they are pairwise distinct). Thus,
\[  2^{|\mathcal{F}_\sB\cap[1,n]|-|\mathcal{F}_{\sB^*}\cap[1,n]|}\leq p_n(\eta). \]
 It follows that
\begin{align}
\label{d-d*_below}
\begin{split}
\overline{d}-\overline{d}^*&=\limsup_{n\to\infty}\frac{|\mathcal{F}_\sB\cap[1,n]|}{n}-\limsup_{n\to\infty}\frac{|\mathcal{F}_{\sB^*}\cap[1,n]|}{n}\\
&\leq\limsup_{n\to\infty}\frac{|\mathcal{F}_\sB\cap[1,n]|-|\mathcal{F}_{\sB^*}\cap[1,n]|}{n}\leq \lim_{n\to\infty}\frac{\log_2 p_n(\eta)}{n}=h(X_\eta,\sigma).
\end{split}
\end{align}
For general (not necessarily taut) \( \sB \), we apply \eqref{d-d*_below} to the tautification \( \sB^{\prime} \). We use \( X_{\eta^{\prime}} \subseteq X_{\eta} \), \( (\sB^{\prime})^{*} = \sB^{*} \) and \( \overline{d}^{\prime} = \overline{d} \) (see~\eqref{j3}, \eqref{zale} and \eqref{jeszczeto}, respectively) to obtain
\[ h( X_{\eta} ) \geq h( X_{ \eta^{\prime} } ) \geq \overline{d}^{\prime} - (\overline{d}^{\prime})^{*} = \overline{d} - \overline{d}^{*} . \]

Now, assume additionally that $X_{\eta^*}$ is uniquely ergodic. Fix $K\geq1$ and let $n\in(\prod_{b\in\sB_K}b)\N$. Since $\eta^*\leq\eta\leq\eta_K$, it follows that
\[
p_n(\eta)\leq p_n(\overline{[\eta^*,\eta_K]}).
\]
For any block $B\in\{0,1\}^n$ which appears in $\overline{[\eta^*,\eta_K]}$, there exists $M\in\Z$ such that 
\[
\eta^*[M+1,M+n]\leq B\leq\eta_K[M+1,M+n].
\] 
Clearly, if $\eta^*( M+\ell )=1$ then $B(\ell)=1$ and there are \[|\supp\eta^*[M+1,M+n]|\] such ``mandatory'' $1$'s on $B$ coming from $\eta^*$. On the other hand, if $\eta_K(M+\ell)=0$ then also $B(\ell)=0$ and there are
\[
n-|\supp\eta_K[M+1,M+n]|
\]
such ``mandatory'' $0$'s on $B$ coming from $\eta_K$. All the other positions on $B$ can be altered arbitrarily, without loosing the property that $B$ appears in $\overline{[\eta^*,\eta_K]}$. The number of such ``free'' positions equals
\[ |\supp\eta_K[M+1,M+n]|-|\supp\eta^*[M+1,M+n]|. \]
Each choice of $0$'s and $1$'s on the ``free'' positions yields a different block of length $n$ from $\overline{[\eta^*,\eta_K]}$. Thus, for each choice of $M$ we obtain
\[
2^{|\supp\eta_K[M+1,M+n]|-|\supp\eta^*[M+1,M+n]|}
\]
blocks and it follows that
\begin{equation}\label{gamma_K}
p_n(\eta)\leq p_n(\overline{[\eta^*,\eta_K]})\leq p_n(\eta^*)\cdot p_n(\eta_K)\cdot2^{\sup_{M\in\Z}(|\supp\eta_K[M+1,M+n]|-|\supp\eta^*[M+1,M+n]|)}.
\end{equation} 
Since $\eta_K$ is $\lcm(\sB_K)$-periodic and $\lcm(\sB_K)\mid n$, we have
\begin{equation}\label{periodic}
|\supp\eta_K[M+1,M+n]|=nd(\mathcal{F}_{\sB_K}).
\end{equation} 
By the uniform ergodicity of $X_{\eta^*}$, for any $\vep>0$ and for large enough $n$, we have
\begin{equation}\label{uniform}
n(d^*-\vep)\leq | \supp\eta^*[M+1,M+n] | \leq n(d^*+\vep)
\end{equation}
for every $M\in\Z$. Using \eqref{gamma_K}, \eqref{periodic} and \eqref{uniform}, we conclude that
\[
p_n(\eta)\leq p_n(\eta^*) \cdot p_n (\eta_K) \cdot 2^{nd(\mathcal{F}_{\sB_K})-n(d^*-\vep)}.
\]
Hence 
\[ h(X_\eta,\sigma)\leq h(X_{\eta^*},\sigma)+h(X_{\eta_K},\sigma)+d(\mathcal{F}_{\sB_K})-d^*.\]
By Remark \ref{entropy_unique_erg}, we have $h(X_{\eta^*},\sigma)=0$ and $h(X_{\eta_K},\sigma)=0$ since $\eta_K$ is periodic. Recall also that by the Davenport-Erd\H{o}s theorem (i.e.~\eqref{DE}), $\lim_{K\to\infty}d(\mathcal{F}_{\sB_K})=\overline{d}$. This yields
\[
h(X_\eta)\leq \overline{d}-d^*,
\]
which completes the proof of Theorem~\ref{entropia}.
\end{proof}

%There are uniquely ergodic Toeplitz subshifts with positive entropy, see \cite{MR1199322}. 

For the proof of Corollary~\ref{zerowanie}, we will need the following lemma.
\begin{lemma}\label{lm4}
For any $\mathscr{B}\subseteq\N$ such that $\eta'\neq \eta^*$, we have $\overline{d}>\overline{d}^*$.
\end{lemma}
\begin{proof}
Since $\overline{d}=\overline{d}':=\overline{d}(\mathcal{F}_{\mathscr{B}'})$, we can assume without loss of generality that $\mathscr{B}$ is taut. Let $(\ell_i)$ be a sequence realizing the lower density of $\mathcal{M}_{\mathscr{B}^*}$. It follows by the result of Davenport and Erd\H{o}s (i.e.\ by~\eqref{DE}) that
\begin{multline*}
\overline{d}^*=\overline{d}(\mathcal{F}_{\mathscr{B}^*})=\lim_{i\to \infty}\frac{1}{\ell_i}|\mathcal{F}_{\mathscr{B}^*}\cap [1,\ell_i]|=\liminf_{i\to \infty}\frac{1}{\ell_i}|\mathcal{F}_{\mathscr{B}^*}\cap [1,\ell_i]| \\
\leq\liminf_{i\to\infty}\frac{1}{\ell_i}|\mathcal{F}_\mathscr{B}\cap [1,\ell_i]|\leq
\limsup_{i\to\infty}\frac{1}{\ell_i}|\mathcal{F}_\mathscr{B}\cap [1,\ell_i]|\leq \overline{d}(\mathcal{F}_\mathscr{B})=\overline{d}.
\end{multline*}
If $\overline{d}=\overline{d}^*$ then all inequalities in the above formula become equalities. In particular, 
\[
\lim_{i\to\infty}\frac{1}{\ell_i}|\mathcal{F}_\mathscr{B}\cap [1,\ell_i]| \text{ exists and equals }\overline{d}=\overline{d}^*,
\]
so that $\eta$ is generic along $(\ell_i)$ for $\nu_\eta$.
Since $\eta^*\leq \eta$ (by the construction of $\mathscr{B}^*$), it follows that $\frac{1}{\ell_i}|\{n\in[1,\ell_i]:\eta( n ) \neq \eta^{*}( n ) \}| \to 0$. Thus, since $\eta^*$ is generic along $(\ell_i)$ for $\nu_{\eta^*}$, it follows immediately that $\eta$ has to be generic along $(\ell_i)$ for the very same measure, i.e.\ $\nu_{\eta^*}$. However, we know that the Mirsky measure $\nu_\eta$ is the unique measure of maximal density, i.e. the invariant measure of the greatest value for the cylinder $\{x\in X_\eta : x(0)=1\}$, in each $\mathscr{B}$-free subshift (see, e.g., Theorem~4 and Corollary~4 in~\cite{MR3784254}, cf. also Chapter~7 in~\cite{MR3136260}), which gives us $\nu_\eta=\nu_{\eta^*}$. Now it suffices to use Corollary~9.2 from~\cite{MR3803141} which says (in particular) that the latter condition is equivalent to $\eta=\eta^*$ (cf.~\eqref{AA}).
\end{proof}

\begin{proof}[Proof of Corollary~\ref{zerowanie}]
By Remark~\ref{entropy_unique_erg}, if $X_\eta$ is uniquely ergodic then $h(X_\eta)=0$.

Suppose that $h(X_\eta)=0$. Then, by Theorem~\ref{entropia}, we have $\overline{d}=d^*$ which implies $\eta'=\eta^*$ by Lemma~\ref{lm4}. The latter condition is equivalent to $\mathscr{B}'=\mathscr{B}^*$ by Theorem~L in~\cite{MR3803141}, as both $\mathscr{B}'$ and $\mathscr{B}^*$ are taut (cf.~\eqref{AA}). It follows immediately that $X_{\eta'}$ must be uniquely ergodic as it is equal to $X_{\eta^*}$ and the latter subshift is uniquely ergodic since $\eta^*$ is assumed to be a regular Toeplitz sequence. It suffices to use Theorem~\ref{wnioC} to complete the proof.
\end{proof}

\subsection{Intrinsic ergodicity of \texorpdfstring{$X_\eta$}{X\_eta}: proof of Theorem~\ref{thme}}\label{wewna}
Consider first the case when $\eta'=\eta^*$. It follows by Theorem~\ref{wnioC} that
\[
\mathcal{P}(X_\eta)=\mathcal{P}(X_{\eta'})=\mathcal{P}(X_{\eta^*}).
\]
Thus, if $X_{\eta^*}$ is uniquely ergodic then $X_\eta$ is also uniquely ergodic. Moreover, the pair $(\eta^*,\eta)$ is quasi-generic for
\(
\nu_{\eta^*}\triangle\nu_\eta
\)
along $(\ell_i)$ realizing the lower density of $\mathcal{M}_\mathscr{B}$. It follows by Remark~\ref{entropy_unique_erg} that
\[
\{\nu_\eta\}=\mathcal{P}(X_\eta)=\mathcal{P}(X_{\eta^*})=\{\nu_{\eta^*}\}.
\]
Thus, $\nu_{\eta^*}\triangle\nu_\eta$ is the diagonal joining of two copies of $\nu_\eta$. Let $(x,x,y)$ be a generic point for $(\nu_{\eta^*}\triangle \nu_\eta)\otimes B_{\nicefrac12,\nicefrac12}$. Then $N(x,x,y)=x$ is a generic point for 
\(
N_\ast((\nu_{\eta^*}\triangle \nu_\eta)\otimes B_{\nicefrac12,\nicefrac12}).
\)
It follows immediately that
\( 
N_\ast((\nu_{\eta^*}\triangle \nu_\eta)\otimes B_{\nicefrac12,\nicefrac12})=\nu_\eta \).

Assume now that $\eta'\neq \eta^*$. We will study the following diagram:
\begin{equation}\label{cov}
\begin{tikzpicture}[baseline=(current  bounding  box.center)]
%punkty
\node (A1) at (0,0) {$(H\times \{0,1\}^\Z,R\times \sigma)$};
\node (B1) at (0,-2) {$(H\times \{0,1\}^\Z, \widetilde{R})$};
\node (C1) at (0,-4) {$([\underline{\varphi},\varphi],\sigma)$};
\draw[->]
	(A1) edge node[auto] {$\Psi$} (B1)
	(B1) edge node[auto] {$\Phi$} (C1);
\draw (A1) -- (3,0) -- node[auto] {$M_H$} (3,-4);
\draw[->] (3,-4) -- (C1);
\end{tikzpicture}
\end{equation}
Let us now introduce all maps appearing in this diagram. We define $\widetilde{R}\colon H\times \{0,1\}^\Z \to H\times \{0,1\}^\Z$ by
\[
\widetilde{R}(h,x)=\begin{cases}
(Rh,x) & \text{ if } \underline{\varphi}(h)(0)=\varphi(h)(0),\\
(Rh,\sigma x) & \text{ if } 0=\underline{\varphi}(h)(0)<\varphi(h)(0)=1.
\end{cases}
\]
Let
\[
Z_\infty:=\{z\in\{0,1\}^\Z : |\text{supp }z \cap (-\infty,0]|=|\text{supp }z \cap [0,\infty)|=\infty\}.
\]
Given $x\in \{0,1\}^\Z$ and $z\in Z_\infty$, let $\hat{x}_z$ be the sequence obtained by reading consecutive coordinates of $x$ which are in the support of $z$ and such that
\[
\hat{x}_z(0)=x(\min\{k\geq 0 : z(k)=1\}).
\]
Now, let
\[
H_\infty:=\{h\in H : R^nh\in W\setminus \underline{W} \text{ infinitely often both in the future and in the past}\}
\]
and define $\Psi\colon H_\infty\times\{0,1\}^\Z\to H_\infty\times \{0,1\}^\Z$ by
\[
\Psi(h,x)=(h,\widehat{x}_{\varphi(h)-\underline{\varphi}(h)})
\]
(notice that $\varphi(h)(n)-\underline{\varphi}(h)(n)=1 \iff R^nh\in W\setminus \underline{W}$, so for $h\in H_\infty$, we have $\varphi(h)-\underline{\varphi}(h)\in Z_\infty$). It remains to define $\Phi$. Let $\Phi\colon H_\infty\times \{0,1\}^\Z\to[\underline{\varphi},\varphi]$ be defined by mapping $ (h , x) $ to the unique element in $[\underline{\varphi},\varphi]$ such that 
\[
\underline{\varphi}(h) \leq \Phi(h,x)\leq \varphi(h) \text{ and} \mkern80mu \widehat{\phantom{x}} \mkern-80mu {({\Phi( h ,x)})}_{\varphi( h )-\underline{\varphi}(h)}=x.
\]
We will show that diagram~\eqref{cov} commutes (it will then follow by Theorem~\ref{glowne} that %it is surjective
the maps $M_H$ and $\Phi$ are surjective morphisms).

\begin{lemma}\label{lm5}
For any $\mathscr{B}\subseteq\N$, we have $m_H(W\setminus \underline{W})=\overline{d}-\overline{d}^*$. Moreover, if $\eta'\neq \eta^*$, we have $m_H(W\setminus \underline{W})>0$.
\end{lemma}
Before we begin the proof of this lemma, recall some results from~\cite{MR4280951} that we already mentioned in the introduction: there is a continuous surjective group homomorphism $\Gamma_{H,H^{*}} \colon H \to H^{*} $, which maps \( \Delta(n) \) to \( \Delta^{*}(n) \). In addition it has the following property, see equation~\eqref{mapsto}:
\begin{equation*}
\Gamma_{H,H^{*}}(\underline{W})=W^* \text{ and }\Gamma_{H,H^{*}}(H\setminus \underline{W})=H^*\setminus W^*.
\end{equation*}
Recall also that it was shown in Lemma~4.1 in~\cite{MR3947636} that 
\begin{equation}\label{miaraW}
m_H(W)=\overline{d}(\mathcal{F}_\mathscr{B})=\overline{d}.
\end{equation}
\begin{proof}[Proof of Lemma~\ref{lm5}]
We have
\begin{align*}
m_H(W\setminus \underline{W})&=m_H(W)-m_H(\underline{W})\\
&=m_H(W)-m_H(\Gamma_{H,H^{*}}^{-1}(W^*))\\
&=m_H(W)-(\Gamma_{H,H^{*}})_{\ast}(m_H)(W^*)\\
&=m_H(W)-m_{H^*}(W^*)
\end{align*}
(the second equality follows from~\eqref{mapsto} and the fourth equality follows by the unique ergodicity of~$R^*$). It remains to use~\eqref{miaraW} to deduce that $m_H(W\setminus \underline{W})=\overline{d}-\overline{d}^*$ and Lemma~\ref{lm4} to conclude that $m_H(W\setminus \underline{W})>0$ whenever $\eta'\neq\eta^*$.
\end{proof}
It follows now from Lemma~\ref{lm5} and from the ergodicity of $(H,R,m_H)$ that $m_H(H_\infty)=1$. Thus, in order to conclude that~\eqref{cov} commutes, it remains to check whether for every $h\in H_\infty$ and every $x\in\{0,1\}^\Z$ we have the commutativity relations
\begin{align}
\label{uuu}
\begin{split}M_H(h,x) &= (\Phi \circ \Psi)(h,x) , \\
(\widetilde{R}\circ\Psi)(h,x) &= (\Psi\circ(R\times \sigma))(h,x) ,\\
(\sigma\circ\Phi)(h,x) &= (\Phi\circ \widetilde{R})(h,x).
\end{split}
\end{align}
The first equality in~\eqref{uuu} is immediate by the definition of the maps, while~the second and third follow from
\begin{equation}\label{haty}
\widehat{\sigma x}_{\sigma z}=\begin{cases}
\hat{x}_z,& \text{ if }z(0)=0,\\
\sigma\hat{x}_z,&\text{ if }z(0)=1,
\end{cases}
\end{equation}
(the proof of~\eqref{haty} consists of a straightforward but lengthy calculation; an analogous property is proved in~\cite{MR3356811}).

Notice that it follows by Theorem~\ref{glowne} that the morphism $M_H$ is surjective. Thus %, in particular, 
the morphism $\Phi$ is also surjective.

Now, we are ready to complete the proof of the intrinsic ergodicity of $X_\eta$. The main ideas come from~\cite{MR3356811}. We will present the sketch of the proof only (similarly as in \cite{MR3803141} for $\widetilde{X}_\eta$). Clearly, any point from $H_\infty \cap (W\setminus \underline{W})$ returns to $W\setminus\underline{W}$ infinitely often under $R$ and $m_H(H_\infty \cap (W\setminus \underline{W}))=m_H(W\setminus\underline{W})$. Recall that
\[
\widetilde{R}(h,x)=\begin{cases}
(Rh,x),& \text{ if } \underline{\varphi}(h)(0)=\varphi(h)(0),\\
(Rh,\sigma x),& \text{ if } 0=\underline{\varphi}(h)(0)<\varphi(h)(0)=1.
\end{cases}
\]
Then every point from $(H_\infty \cap (W\setminus \underline{W}))\times \{0,1\}^\Z$ returns to $(W\setminus\underline{W})\times \{0,1\}^\Z$ infinitely often under $\widetilde{R}$ and
$\nu((H_\infty \cap (W\setminus \underline{W}))\times\{0,1\}^\Z)=\nu((W\setminus\underline{W})\times \{0,1\}^\Z)$ for every $\nu \in \mathcal{P}(H\times \{0,1\}^\Z,\widetilde{R})$. Thus, the induced transformation $\widetilde{R}_{{(W\setminus \underline{W})}\times \{0,1\}^\Z}$ is well-defined, i.e. $\widetilde{R}_{{(W\setminus \underline{W})}\times \{0,1\}^\Z}(h,x)=\widetilde{R}^{n(h,x)}(h,x)$ for $\nu$-a.e. $(h,x)\in (W\setminus \underline{W})\times\{0,1\}^\Z$, where $n(h,x):=\min\{n\geq1: \ \widetilde{R}^n(h,x)\in (W\setminus \underline{W})\times\{0,1\}^\Z\}$. 
It follows that $\widetilde{R}_{(W\setminus \underline{W})\times \{0,1\}^\Z}=R_{(W\setminus \underline{W})}\times \sigma$ a.e.\ for any $\widetilde{R}$-invariant measure.

We will show now that $\widetilde{R}$ has a unique measure of maximal (measure-theoretic) entropy. Since $m_H(W\setminus\underline{W})>0$ whenever $\eta'\neq\eta^*$ (see Lemma~\ref{lm5}) and since $\kappa(W\setminus\underline{W}\times\{0,1\}^\Z)=m_H(W\setminus\underline{W})>0$ for any $R_{W\setminus\underline{W}}\times\sigma$-invariant probability measure $\kappa$, in view of Abramov's formula it suffices to show that $\widetilde{R}_{(W\setminus \underline{W})\times \{0,1\}^\Z}=R_{(W\setminus \underline{W})}\times \sigma$ has a unique measure of maximal entropy. For any $R_{(W\setminus \underline{W})}\times\sigma$-invariant measure $\kappa$, by the Pinsker formula, we have
\begin{align*}
\begin{split}
 &h(\{0,1\}^\Z,\sigma,(\pi_2)_\ast(\kappa))\leq h( (W\setminus \underline{W})\times \{0,1\}^\Z,R_{W\setminus \underline{W}}\times \sigma,\kappa )\\
&\quad\leq h(W\setminus \underline{W}, R_{W\setminus \underline{W}},(\pi_1)_\ast(\kappa)|_{(W\setminus \underline{W})}) +h(\{0,1\}^\Z,\sigma,(\pi_2)_\ast(\kappa))=h(\{0,1\}^\Z,\sigma,(\pi_2)_\ast(\kappa)),
\end{split}
\end{align*}
where $h(W\setminus \underline{W}, R_{W\setminus \underline{W}},(\pi_1)_\ast(\kappa)|_{(W\setminus \underline{W})})$ vanishes by Abramov's formula as $R_{W\setminus\underline{W}}$ is an induced map coming from a rotation.
Since $(\pi_2)_\ast(\kappa)$ can be arbitrary, it follows that a measure $\kappa$ has the maximal entropy among all $R_{W\setminus \underline{W}}\times\sigma$-invariant measures if and only if $h((W\setminus \underline{W})\times\{0,1\}^{\Z}, R_{W\setminus \underline{W}}\times\sigma,\kappa)=h(\{0,1\}^\Z,\sigma)$. Moreover, $\kappa$ is a measure of maximal entropy for $R_{W\setminus \underline{W}}\times\sigma$ if and only if $(\pi_2)_\ast(\kappa)$ is the measure of maximal entropy for $\sigma$, that is, when $(\pi_2)_\ast(\kappa)$ is the Bernoulli measure $B_{\nicefrac12,\nicefrac12}$, that is, when $\kappa$ is a joining of the unique invariant measure for $R_{W\setminus \underline{W}}$ and $B_{\nicefrac12,\nicefrac12}$. Since the unique invariant measure for $R_{W\setminus \underline{W}}$ is of zero entropy, it follows from the disjointness of K-automorphisms with zero entropy automorphisms~\cite{MR0213508} that $\kappa$ is the product measure. In particular, $\kappa$ is unique.

The last step to conclude the intrinsic ergodicity of $X_\eta$ is to show that 
\begin{equation}\label{eq:jan}
h(H\times \{0,1\}^\Z, \widetilde{R})=m_H(W\setminus\overline{W})=\overline{d}-d^*=h(X_\eta).
\end{equation}
Let us justify each of the equalities above. By the variational principle, by the Abramov's formula and by the Pinsker formula, we have
\begin{align*}
h(H\times \{0,1\}^\Z, \widetilde{R})&=\sup_\rho\{h(H\times \{0,1\}^\Z, \widetilde{R},\rho) \}\\
&=m_H(W\setminus\underline{W})\cdot\sup_\rho\{h((W\setminus \underline{W})\times \{0,1\}^\Z,R_{W\setminus\underline{W}}\times\sigma,\rho)\}\\
&=m_H(W\setminus \underline{W})\cdot \sup_\rho \{h(\{0,1\}^\Z,\rho)\}=m_H(W\setminus\underline{W}),
\end{align*}
where the suprema are taken over all Borel probability invariant measures for the corresponding maps. This yields the first equality in~\eqref{eq:jan}. Moreover, the middle equality in~\eqref{eq:jan} follows by Lemma~\ref{lm5}, while the last one follows by Theorem~\ref{entropia}.

\subsection{Entropy density of \texorpdfstring{$X_\eta$}{X\_eta}: proof of Theorem~\ref{dominik}}

The idea of the proof of Theorem~\ref{dominik} is the same as that of the analogous result for $\widetilde{X}_\eta$ in~\cite{MR4544150}. Let us introduce the necessary tools and notation. Given $x,y\in\{0,1\}^\Z$, consider the following \emph{premetric}:
\[
\underline{d}(x,y):=\liminf_{n\to\infty}\frac{1}{n}|\{1\leq i\leq n : x( i ) \neq y( i ) \}|
\]
(being a premetric means that $\underline{d}$ is a real-valued, nonnegative, symmetric function on $(\{0,1\}^\Z)^2$ vanishing on the diagonal; the triangle inequality for $\underline{d}$ fails). As a premetric, $\underline{d}$ induces a Hausdorff pseudometric $\underline{d}^H$ on the space of all nonempty subsets of $\{0,1\}^\Z$ in the following way: 
\[
\underline{d}(x,Y):=\inf_{y\in Y}\underline{d}(x,y) \text{ and }\underline{d}^H(X,Y):=\max\{\sup_{x\in X}\underline{d}(x,Y),\sup_{y\in Y}\underline{d}(y,X)\}
\]
for any $\emptyset\neq X, Y\subseteq \{0,1\}^\Z$ and $x\in X,y\in Y$.

Let us now recall some results from~\cite{MR4544150} (we formulate them for $0$-$1$ shifts, however, they are valid for shifts over any finite alphabet).

\begin{proposition}[Proposition 26 in~\cite{MR4544150}]\label{sofic}
Let $x\in\{0,1\}^\Z$ be a periodic point under $\sigma$. Then the hereditary closure of the orbit of $x$ is a transitive sofic shift.
\end{proposition}
\begin{remark}\label{sofic1}
We skip here the definition of a sofic shift as it is quite technical and this notion serves here as a tool only. Namely, in any sofic transitive shift the ergodic measures are entropy dense (more general results are known, see~\cite{MR2109476,MR1291239}). The following modification of Proposition~\ref{sofic} holds. Let $w, x\in\{0,1\}^\Z$ be periodic, such that $w\leq x$.  Then $[w,x]$ is a transitive sofic shift. The proof is a straightforward adjustment of the proof of Proposition~26 in~\cite{MR4544150}.
\end{remark}
\begin{proposition}[Corollary 20 in~\cite{MR4544150}]\label{z:1}
Let $(X_K)_{K\geq 1}\subseteq \{0,1\}^\Z$ be a sequence of transitive sofic shifts. If $X\subseteq\{0,1\}^\Z$ is a subshift such that $\underline{d}^H(X_K,X)\to 0$ then ergodic measures are entropy-dense in $\mathcal{P}(X)$.
\end{proposition}

\begin{proof}[Proof of Theorem~\ref{dominik}]
Since $\mathcal{P}(X_\eta)=\mathcal{P}(X_{\eta'})$ by Theorem~\ref{wnioC}, we can assume without loss of generality that $\mathscr{B}$ is taut. Hence \( X_{\eta} = \overline{[ \eta^{*} , \eta ]} \) by Proposition~\ref{between_etas}, and in view of Remark~\ref{sofic1} and Proposition~\ref{z:1}, it suffices to prove that \( \underline{d}^{H}( [\underline{\eta}_{K} , \eta_{K} ] , \overline{[ \eta^{*} , \eta ]} ) \to 0 \). To do so, we show that
\[ \underline{d}^{H}( [\underline{\eta}_{K} , \eta_{K} ] , \overline{[ \eta^{*} , \eta ]} ) \leq \overline{d}( \underline{\eta}_{K} , \eta^{*} ) + \underline{d}( \eta , \eta_{K} ) , \]
where the right hand side tends to zero by the regularity of $\eta^*$ (cf.~\eqref{nowe}) and the Davenport-Erd\H{o}s theorem (i.e.~\eqref{DE}). Fix now \(K\geq 1\). We claim that for \( \underline{\eta}_{K} \leq x \leq \eta_{K} \) there exists \( y \in [ \eta^{*} , \eta ] \) with
\begin{equation}\label{nowee} 
\underline{d}( x , y ) \leq \overline{d}( \underline{\eta}_{K} , \eta^{*} ) + \underline{d}( \eta , \eta_{K} ) . 
\end{equation} 
Indeed, set \( y := N( \eta^{*} , \eta , x ) = \eta^{*} + x( \eta - \eta^{*} ) \). Then \( y(n) \neq x(n) \) implies that \( \eta^{*}( n ) = \eta( n ) \) and \( \underline{\eta}_{K}( n ) \neq \eta_{K}( n ) \) (recall that \( \underline{\eta}_{K} \leq \eta^{*} \leq \eta \leq \eta_{K} \)). Thus, for every \( n \in \Z \) with \( y(n) \neq x(n) \) we have either that \( \underline{\eta}_{K}( n ) \neq \eta^{*}( n ) \) or that \( \eta( n ) \neq \eta_{K}( n ) \), i.e.,
\[
\{n\in \mathbb{N} : x(n)\neq y(n)\}\subseteq \{n\in\mathbb{N} : \underline{\eta}_K(n)\neq \eta^*(n)\}\cup\{n\in\mathbb{N} : \eta(n)\neq \eta_K(n)\}.
\]
This yields~\eqref{nowee}. To finish the proof, note that for every \( x \in  [\underline{\eta}_{K} , \eta_{K} ] \) there exists \( m \in \Z \) such that \( \underline{\eta}_{K} \leq \sigma^{m} x \leq \eta_{K} \). By the above construction, for every \( x \in  [\underline{\eta}_{K} , \eta_{K} ] \) there exists therefore \( y \in [ \eta^{*} , \eta ] \) with
\( \underline{d}( x , y ) \leq \underline{d}( \underline{\eta}_{K} , \eta^{*} ) + \underline{d}( \eta , \eta_{K} ) \). In addition, we have \( \overline{[ \eta^{*} , \eta ]} \subseteq  [\underline{\eta}_{K} , \eta_{K} ] \).
\end{proof}

\subsection*{Acknowledgements}
We are indebted to Stanis\l{}aw Kasjan and Mariusz Lema\'{n}czyk for helpful discussions, in particular  on category theory notation. Research of the first two authors is supported by Narodowe Centrum Nauki grant UMO-2019/33/B/ST1/00364. Research of the third author is funded by the Deutsche Forschungsgemeinschaft (DFG, German Research Foundation) -- Projektnummer 454053022.

%\bibliographystyle{acm}
%\bibliography{basic}

\bigskip
\footnotesize
\noindent
Aurelia Dymek, Joanna Ku\l aga-Przymus, Daniel Sell\\
\textsc{Faculty of Mathematics and Computer Science}\\
\textsc{Nicolaus Copernicus University}\\ 
\textsc{Chopina 12/18, 87-100 Toru\'{n}, Poland}\par\nopagebreak
\noindent
\texttt{\{aurbart,joasiak,dsell\}@mat.umk.pl}

\end{document}